\documentclass[final,leqno,onefignum,onetabnum]{siamltex1213}
\usepackage{amsfonts}

\usepackage{amsmath,booktabs,ctable,threeparttable}
\usepackage{amssymb,amsfonts,boxedminipage}
\usepackage{multirow}
\usepackage{algorithmic,algorithm}
\usepackage[caption=false]{subfig}

\title{On Inner Iterations of Jacobi-Davidson Type Methods for Large SVD
Computations\thanks{This work was supported in part by
the National Science Foundation of China (No.11771249).}}

\author{Jinzhi Huang\thanks{Department of Mathematical Sciences, Tsinghua University, 100084 Beijing,
China(\email{huangjz15@mails.tsinghua.edu.cn}).}
\and  Zhongxiao Jia\thanks{Department of Mathematical Sciences, Tsinghua University, 100084
Beijing, China(\email{jiazx@tsinghua.edu.cn}).}}

\begin{document}

\maketitle
% REQUIRED
\begin{abstract}
We make a convergence analysis of the harmonic and refined
harmonic extraction versions of Jacobi-Davidson SVD (JDSVD) type methods
for computing one or more interior singular triplets of a large matrix $A$.
At each outer iteration of these methods, a correction equation,
i.e., inner linear system, is solved
approximately by using iterative methods, which leads to two inexact JDSVD type
methods, as opposed to the exact methods where correction equations
are solved exactly. Accuracy of inner iterations critically affects
the convergence and overall efficiency of the inexact
JDSVD methods. A central problem is how accurately
the correction equations should be solved so as to ensure
that both of the inexact JDSVD methods can mimic their exact counterparts
well, that is, they use almost the same outer iterations to achieve
the convergence. In this paper, similar to the available results on
the JD type methods for
large matrix eigenvalue problems, we prove that each inexact JDSVD method
behaves like its exact counterpart if all the correction equations are solved
with $low\ or\ modest$ accuracy during outer iterations.
Based on the theory, we propose practical stopping criteria for inner iterations.
Numerical experiments confirm our theory and the effectiveness of the inexact
algorithms.
\end{abstract}
%
%% REQUIRED
%
%% REQUIRED
\pagestyle{myheadings}
\thispagestyle{plain}
\markboth{Jinzhi Huang and Zhongxiao Jia}{Inner Iterations of
Jacobi-Davidson Type Methods}

\section{Introduction}\label{introduction}
Consider the singular value decomposition (SVD)
\begin{equation}
 \left\{\begin{aligned}&Av_i=\sigma_i u_i\\&A^Tu_i
 =\sigma_i v_i\end{aligned}\right.,  \quad  \quad i=1,2,\dots,N
 \end{equation}
of a large and possibly sparse matrix $A\in\mathbb{R}^{M\times N}$
with $M\geq N$.
The $(\sigma_i, u_i, v_i),\ i=1,2,\ldots,N$, are called the singular triplets of $A$.
For a given target $\tau>0$, let the singular values $\sigma_1, \sigma_2,\dots,
\sigma_N$ be labeled as
\begin{equation}\label{equation:43}
|\sigma_1-\tau|<|\sigma_2-\tau|\leq\dots\leq|\sigma_N-\tau|.
\end{equation}
We are interested in the simple singular value $\sigma_1$ closest to the target
$\tau$ and/or the corresponding left and right singular vectors $u_1$ and $v_1$.
We will denote $(\sigma_1, u_1,v_1)$ by $(\sigma, u^{*},v^{*})$ for simplicity,
which is called an interior singular triplet when $\tau$ is inside the singular
spectrum of $A$. We make two remarks. The first is that all the algorithms to
be proposed, the
analysis and results apply to a complex $A$ trivially with the transpose of a vector
or matrix replaced by its conjugate transpose. The second is
that we focus on the analysis of the algorithms for computing one simple
singular triplet of $A$ and then derive the algorithms with deflation
for computing more than one, i.e.,  $\ell>1$, singular triplets of $A$
and extend our analysis and results to these variants.

Projection methods that are widely used for large eigenproblems
\cite{golub2012matrix,parlett1998symmetric,saad2011numerical,stewart2001matrix}
have been adapted to the computation of
$(\sigma, u^{*},v^{*})$. Given projection subspaces,
there are four extraction approaches: standard, harmonic, refined, and refined harmonic
extractions.
The standard extraction \cite{hernandes2008robust,hochstenbach2001jacobi,
hochstenbach2004harmonic,jia2003implicitly,simon2000lowrank,stoll2012krylov}
suits well for large singular values (cf. Theorem 4.3 of \cite{hochstenbach2001jacobi}).
The harmonic extraction is preferable to the interior and the smallest
singular values. However, the approximate singular vectors obtained by the standard
and harmonic extractions may converge irregularly or even fail to converge
even though the approximate singular values converge \cite{jia2015harmonic,jia2003implicitly}.
The refined and refined harmonic extractions \cite{hochstenbach2001jacobi,
hochstenbach2004harmonic,jia2002refinedharmonic,jia2005convergence,
jia2003implicitly,jia2010refined,
kokiopoulou2004computing,niu2012harmonic,wu17}
fix the possible non-convergence of singular
vectors and obtain more accurate approximations. As observed and claimed in
\cite{hochstenbach2004harmonic,jia2003implicitly, jia2010refined,kokiopoulou2004computing,
wu16,wu2015preconditioned}, the refined and refined harmonic extractions
appear to give better accuracy than the standard and harmonic counterparts,
respectively.

As a matter of fact, when computing a truly interior singular triplet $(\sigma, u^{*},v^{*})$,
the standard extraction encounters a serious difficulty even if
searching subspaces are sufficiently good:
it is practically hard to pick up a correct Ritz value to approximate $\sigma^*$
even though there is a good one among them.
As a result, whenever a wrong Ritz value is selected, the refined extraction
definitively fails to deliver a good approximation to $(u^*,v^*)$.
Therefore, the refined extraction may also be unsuitable to compute the truly
interior $(\sigma, u^{*},v^{*})$.
In contrast, the above two difficulties can be nicely overcome by using the
harmonic and refined harmonic extractions, especially by the latter one
because the refined harmonic extraction will produce better approximate
singular vectors whose convergence can be guaranteed provided that searching
subspaces are sufficiently accurate \cite{jia2015harmonic,wu17}.
In other words, unlike the standard and refined extractions,
for good searching subspaces, the harmonic and refined harmonic extractions
produce good approximations to the interior $(\sigma^*,u^*,v^*)$,
and the refined harmonic extraction is favorable due to its better convergence.
Because of these reasons, we will only consider the harmonic and refined harmonic
extractions in this paper.

Sleijpen and van der Vorst \cite{sleijpen2000jacobi} propose the
Jacobi-Davidson (JD) method for large matrix eigenvalue problems,
which uses the standard and harmonic extractions.
Hochstenbach \cite{hochstenbach2001jacobi,hochstenbach2004harmonic}
extends the JD method to the SVD computation in a novel way, referred as JDSVD,
which requires that an approximate singular triplet satisfies a certain
double orthogonality condition other than the ordinary orthogonality condition.
Moreover, JDSVD expands the projection subspace in a different way from the JD
applied to the eigenvalue problem of the augmented matrix
$\begin{bmatrix} \begin{smallmatrix}0 &A\\A^T&0\end{smallmatrix} \end{bmatrix}$
directly, whose eigenvalues are $\pm\sigma_i,\ i=1,2,\ldots,N$ and $M-N$ zeros,
and at each iteration the subspace dimension of JDSVD is increased by two
rather than one as done in the latter.
Obviously, JDSVD is mathematically different from the JD method applied to the
augmented matrix directly, and is shown to be more suitable and more effective
than the direct application of the JD method for the eigenvalue problem
to the mentioned augmented matrix.
Wu and Stathopoulos \cite{wu2015preconditioned} propose a two-stage SVD method,
called PHSVD, which transforms the SVD computation into the eigenvalue problems
of $A^TA$ and $\begin{bmatrix} \begin{smallmatrix}0 &A\\A^T&0\end{smallmatrix} \end{bmatrix}$.
Unlike JDSVD, PHSVD performs the projection first on the eigenvalue problem
of the cross-product $A^TA$ and then uses the JD type methods to solve
the eigenvalue problem of
$\begin{bmatrix} \begin{smallmatrix}0 &A\\A^T&0\end{smallmatrix} \end{bmatrix}$.
At each iteration of JDSVD or the second stage of PHSVD, one needs to solve a
correction equation when expanding searching subspaces.
Wu, Romero and Stathopolous \cite{wu16} develop the software PRIMME\_SVDS
that implements the PHSVD method using the standard extraction
and the refined or refined harmonic extraction, and they provide intuitive user
interfaces in C, MATLAB, Python, and R for both ordinary and advanced users
so as to fully exploit the power of PRIMME\_SVDS.

Since the matrix $A$ is large, for each method
it is generally impractical to solve the correction equation by direct solvers.
Therefore, the correction equation has to be solved approximately by iterative
solvers. This gives rise to the inexact JD type methods for either
eigenvalue problems or SVD computations, which distinguish
from the exact methods where all the correction equations are solved exactly.
Accuracy of inner iterations plays the critical role in
the convergence and overall efficiency of the inexact JD methods.
Naturally, a central question is:
how accurately the correction equations should be solved in order to ensure
that they use almost the same outer iterations to achieve
the desired convergence?

Accuracy issues on inner iterations in the JD type methods are fundamental
and have been receiving intensive attention for the JD
type methods for large matrix eigenvalue problems since the advent of the JD method.
Notay \cite{notay2002combination} makes an analysis on the
{\em simplified} JD method, i.e., without subspace acceleration,
for the eigenvalue problem, called JDCG algorithm, in which the correction
equation is solved by the preconditioned CG method and the correction
vector is directly added to the current approximate eigenvector to form a new
approximation in the next outer iteration. He shows that during inner
iterations, the convergence of the outer iteration towards the desired
eigenpair can be monitored by the reduction of the residual norms of inner
iterations. Based on this result, Notay proposes stopping criteria for
inner iterations of the JDCG algorithm.
Hochstenbach and Notay \cite{hochsten2009} extend the results
in \cite{notay2002combination} to the generalized Hermitian eigenvalue
problem. Wu, Romero and Stathopolous
\cite{wu16,wu2015preconditioned}
adapt the stopping criteria in \cite{notay2002combination} for inner iterations to
their general JD type methods for large SVD computations, in which they use
the quasi minimal residual (QMR) algorithm other than the CG algorithm to
solve generally indefinite correction equations.
The JDSVD methods in \cite{hochstenbach2001jacobi,hochstenbach2004harmonic}
simply stop inner iterations after a small number of iterations
are performed.

Unfortunately, Notay's analysis approach and the results in
\cite{hochsten2009,notay2002combination,wu16,wu2015preconditioned}
are only applicable to the simplified JD type methods, and they cannot be
extended to the more complicated and practical {\em general} JD type methods
with subspace acceleration for the eigenvalue problem or the SVD computation,
in which the (approximate) solution of the correction equation is used
to expand the subspace and new approximate eigenpair or singular triplet
is formed by performing one of the four extractions onto the expanded
subspace. However, one must be aware that the simplified JD method is seldom
used practically due to its low effectiveness.
For the general JD type methods for eigenvalue problems that are
based on standard, harmonic and refined harmonic extractions,
Jia and Li \cite{jia2014inner,jia2015harmonic} prove for the first time that once
all the correction equations are solved only with $low\ or\ modest$
accuracy by any iterative solver,
the JD type methods can mimic their exact counterparts well,
in the sense that the outer iterations used by the former ones
to achieve the convergence are almost the same as those used by the latter ones.
Based on the theory established, they propose practical stopping criteria for inner
solvers. A remarkable fact is that,
whenever the JD type methods for the SVD computation
are mathematically equivalent to the JD type methods
applied to the eigenvalue problem of
$\begin{bmatrix} \begin{smallmatrix}0 &A\\A^T&0\end{smallmatrix} \end{bmatrix}$,
all the results in \cite{jia2014inner,jia2015harmonic} directly carry over to
them, which include the algorithms in \cite{wu16,wu2015preconditioned}.

However, the results in
\cite{jia2014inner,jia2015harmonic} have not yet been extended
to the double orthogonality based JDSVD methods
in \cite{hochstenbach2001jacobi,hochstenbach2004harmonic}. It is unclear
whether or not similar results exist and are
applicable to these JDSVD methods. Therefore, reliable and general-purpose
stopping criteria for corresponding inner iterations are still lacking.

In this paper, by requiring that an approximate singular
triplet satisfies the double orthogonality condition \cite{hochstenbach2001jacobi},
we propose a harmonic extraction based JDSVD type method, in which
the approximate singular value must be a certain Rayleigh quotient other
than the associated harmonic Ritz value.
Then, by combining the harmonic extraction with the refined extraction, we present
a refined harmonic extraction based JDSVD type method.
Our harmonic and refined harmonic extractions are
mathematically different from the ones in
\cite{wu16,wu2015preconditioned} in that (i) ours require double orthogonality
condition while theirs use the ordinary orthogonality and (ii) ours expand
the projection subspace by two dimensions while theirs
expand the projection subspace by one dimension at each outer iteration.
Our methods and their derivations
are also different from the harmonic and refined extraction based
JDSVD methods in \cite{hochstenbach2001jacobi,hochstenbach2004harmonic}
in several aspects; for example,
the latter ones use the principle of harmonic and refined extractions
in different manners. On the other hand, we will notice
a key fact that the correction equations involved in our methods
and the ones in \cite{hochstenbach2001jacobi,hochstenbach2004harmonic}
at each outer iteration have the {\em same forms}.
Therefore, the problem on the solution accuracy of these correction equations
can be considered in a {\em unified} way.

Precisely, our primary goal is to establish the {\em lowest} accuracy
requirement for inner iterations involved in JDSVD methods
such that each inexact method mimics its exact counterpart well,
that is, they use very comparable outer iterations to achieve the
desired convergence. The goal is fundamental. First,
with the exact JDSVD type methods as a reference standard,
it guarantees the robustness and generality of the corresponding inexact
counterparts. Second, under this presupposition, the inexact JDSVD type algorithms
are the most efficient since, for a given iterative solver,
we can use the least cost, i.e., fewest inner iterations, to solve each
correction equation at every outer iteration.

Inspired by the work of Jia and Li \cite{jia2014inner,jia2015harmonic},
under the assumption that the current
approximate singular triplet is reasonably good,
we make a rigorous one-step analysis on the two inexact JDSVD type algorithms.
We first establish an intimate connection between the solution accuracy of the
correction equation and the accuracy of the expansion vectors. Then
we derive the accuracy relationship between the inexact and exact expanded
subspaces, based on which we establish the accuracy requirement
for inner iterations. The results show that each inexact JDSVD mimics the
corresponding exact JDSVD well once all the correction equations are solved
with the low or modest accuracy $10^{-4}\sim 10^{-3}$ during outer iterations.
These results are similar to those on the inexact JD type methods for general
eigenvalue problems \cite{jia2014inner,jia2015harmonic}. Furthermore, we prove
that these results can be extended to the JDSVD type methods with deflation
that are used to compute $\ell>1$ singular triplets of $A$.
We consider practical issues and design practical stopping criteria for
inner iterations in the JDSVD methods. It is important to remind
that all the results
hold for the proposed methods in this paper and the methods
in \cite{hochstenbach2001jacobi,hochstenbach2004harmonic} because the correction
equations in them have the same form, as mentioned previously.
Finally, we report numerical
experiments to justify that our theory works well.

The paper is organized as follows. In Section \ref{section:2},
we describe the harmonic and refined harmonic JDSVD methods.
In Section \ref{section:3}, we derive relationships between
the solution accuracy of the correction equation and the accuracy
of the expansion vectors. We prove that the inexact JDSVD methods
can mimic the exact JDSVD methods well when the solution accuracy of
the correction equation is low or modest, i.e., $10^{-4}\sim 10^{-3}$,
during all outer iterations. In Section~\ref{section:7}, we describe the JDSVD type
methods with deflation for computing more than one singular triplets and
extend our results to this case. In Section \ref{section:4},
we design practical stopping criteria for inner iterative solvers.
In Section \ref{section:5}, we report numerical experiments to
confirm our theory. Finally, we conclude the paper in Section \ref{section:6}.

Throughout the paper, denote by $\|\cdot\|$ the 2-norm and $\kappa(X)=
\sigma_{\max}(X)/\sigma_{\min}(X)$ the condition number of a matrix $X$
with $\sigma_{\max}(X)$ and $\sigma_{\min}(X)$ being the largest and
smallest singular values of $X$, by $I_k$ the identity matrix
of order $k$, and by $X^T$ the transpose
of $X$.

\section{Harmonic and refined harmonic JDSVD methods} \label{section:2}

Assume that we have a pair of $m$-dimensional searching
subspaces $\mathcal{U}$ and $\mathcal{V}$, from which we compute
an approximate singular triplet $(\theta,u,v)$ satisfying
$u\in\mathcal{U},v\in\mathcal{V}$ and $||u||=||v||=1$.
Define the residual of $(\theta,u,v)$ as
\begin{equation}\label{residual}
r=r(\theta, u,v):=\begin{bmatrix}Av-\theta u\\ A^Tu-\theta v\end{bmatrix}.
\end{equation}
We require that $r\perp \perp (u, v)$, where $\perp\perp$
means that $Av-\theta u\perp u$ and $A^Tu-\theta v\perp v$,
which is called the double orthogonality in \cite{hochstenbach2001jacobi}.
We mention that the so called ordinary orthogonality condition
means $r\perp [u^T,v^T]^T$.

Assume that the columns of $U$ and $V$ form orthonormal bases of
$\mathcal{U}$ and $\mathcal{V}$, respectively, and
$(\theta,u,v)$ is reasonably good but not converged yet.
Then a JDSVD type method
expands $\mathcal{U}$ and $\mathcal{V}$ in the following way:

Solve the correction equation
\begin{equation}\label{equation:1}
  \begin{bmatrix} I_M-uu^T & \\ & I_N-vv^T \end{bmatrix}
\begin{bmatrix} -\tau I_M  & A\\A^T& -\tau I_N \end{bmatrix}
\begin{bmatrix} I_M-uu^T & \\ & I_N-vv^T\end{bmatrix}
\begin{bmatrix} s\\t \end{bmatrix}=-r
\end{equation}
for $(s,t)\perp\perp (u,v)$, where $s\in \mathbb{R}^M$ and $t\in\mathbb{R}^N$.
Orthonormalize $s$ and $t$ against $U$ and $V$ to get $u_{+}$
and $v_{+}$, respectively. The normalized subspace
expansion vectors are
\begin{equation}\label{exvector}
u_{+}=\frac{(I_M-P_U)s}{\|(I_M-P_U)s\|},
\quad  \quad
v_{+}=\frac{(I_N-P_V)t}{\|(I_N-P_V)t\|},
\end{equation}
where $P_U=UU^T$ and $P_V=VV^T$ are the orthogonal projectors
onto $\mathcal{U}$ and $\mathcal{V}$, respectively. The columns
of $U_{+}=[U, u_{+}]$ and $V_{+}=[V,v_{+}]$ are the orthonormal bases of
the expanded subspaces $\mathcal{U}_+$ and $\mathcal{V}_+$, from which
we can compute new approximate singular triplet.

\subsection{Harmonic Jacobi-Davidson SVD method}\label{subsection:HJDSVD}

It is known that $(\sigma,w^{*}=\frac{1}{\sqrt2}\begin{bmatrix}\begin{smallmatrix}
u^{*}\\v^{*}\end{smallmatrix}\end{bmatrix})$
is an eigenpair of the augmented matrix
\begin{align}\label{augment}
C&=\begin{bmatrix}%\begin{smallmatrix}
0&A\\A^T&0%\end{smallmatrix}
\end{bmatrix}.
\end{align}
We use the harmonic extraction
described in \cite{jia2015harmonic,niu2012harmonic}
to compute an approximation of $(\sigma,w^*)$. Notice that the projection or searching
subspace is spanned by the columns of
$\begin{bmatrix}\begin{smallmatrix}U&\\&V\end{smallmatrix}\end{bmatrix}$.
Given the target $\tau>0$, let the approximate eigenvector be
$\widetilde{w}=\begin{bmatrix}\begin{smallmatrix}U&\\&V\end{smallmatrix}\end{bmatrix}
\widetilde f\in\mathbb{R}^{M+N}$, and define
\begin{equation}\label{defG}
G=\begin{bmatrix} U^T&\\&V^T \end{bmatrix}
\begin{bmatrix} -\tau I_M&A\\
A^T&-\tau I_N \end{bmatrix}^2
\begin{bmatrix} U&\\&V\end{bmatrix}
\end{equation}
and
\begin{equation}\label{defF}
F=\begin{bmatrix} U^T&\\&V^T \end{bmatrix}
\begin{bmatrix} -\tau I_M&A\\A^T&-\tau I_N \end{bmatrix}
\begin{bmatrix} U&\\&V \end{bmatrix}
=\begin{bmatrix} -\tau I_m &H\\ H^T&-\tau I_m \end{bmatrix}
\end{equation}
with $H=U^TAV$. Then the harmonic extraction on $C$ amounts to solving the
following generalized symmetric eigenvalue problem of the matrix pencil $(F, G)$:
\begin{equation}\label{equation:17}\left\{\begin{aligned}
F\widetilde f&=\frac{1}{\nu}G\widetilde f \ \mbox{with}\ \|\widetilde f\|=1\\
\vartheta&=\nu+\tau
\end{aligned}\right.\end{equation}
for the largest eigenvalue $\frac{1}{\nu}$ in magnitude and the
corresponding eigenvector $\widetilde f=\begin{bmatrix}\begin{smallmatrix}
\widetilde c^T,\widetilde d^T\end{smallmatrix}
\end{bmatrix}^T$, where $\widetilde c,\widetilde d\in\mathbb{R}^m$.
We take ($|\vartheta|,\widetilde u=U\frac{\widetilde c}{\|\widetilde c\|},\
\widetilde v=V\frac{\widetilde d}{\|\widetilde d\|})$ to
approximate $(\sigma,u^{*},v^{*})$ if $\widetilde u^T A\widetilde v\geq 0$;
otherwise, we take ($|\vartheta|,\widetilde u=U\frac{\widetilde c}{\|\widetilde c\|},\
\widetilde v=-V\frac{\widetilde d}{\|\widetilde d\|})$.

If $(|\vartheta|,\widetilde u,\widetilde v)$ has not yet converged,
we need to expand the subspaces $\mathcal{U}$ and $\mathcal{V}$.
Let $u=\widetilde u$ and $v=\widetilde v$,
and replace $|\vartheta|$ by the Rayleigh quotient
\begin{equation}\label{equation:30}
\theta=\rho=|\widetilde u^TA\widetilde v|=
\frac{|\widetilde c^TU^TAV\widetilde d|}{\|\widetilde c\|
\|\widetilde d\|}=\frac{|\widetilde c^TH\widetilde d|}
{\|\widetilde c\|\|\widetilde d\|}.
\end{equation}
Then we obtain a consistent equation \eqref{equation:1} with
$r=r(\rho,\widetilde u,\widetilde v)$. We use the solution vectors
$s$ and $t$ to compute the expansion vectors
$u_+$ and $v_+$ defined by \eqref{exvector} and obtain
$\mathcal{U}_{+}$ and $\mathcal{V}_{+}$. The new projection subspace
is spanned by the columns of
$\begin{bmatrix}\begin{smallmatrix}U_+&\\&V_+
\end{smallmatrix}\end{bmatrix}$, and its dimension is thus increased by two.
The resulting method is called the harmonic Jacobi-Davidson SVD (HJDSVD) method.

We mention in passing that in the JD type methods applied to the eigenvalue
problem of $\begin{bmatrix} \begin{smallmatrix}0 &A\\A^T&0
\end{smallmatrix} \end{bmatrix}$ directly,
the current projection subspace is spanned by the orthonormal columns
of $[U^T,V^T]^T$
and the expanded one is spanned by those of $[U_+^T,V_+^T]^T$,
whose dimension
is increased by one. Another mathematically essential distinction is that the
whole expanding vector $[u_+^T,v_+^T]^T$ is obtained by orthonormalizing
$[s^T,t^T]^T$ against the column orthonormal $[U^T,V^T]^T$,
where $[s^T,t^T]^T$ is the solution to a corresponding correction
equation in the {\em different} form from \eqref{equation:1}.

It is easily justified by the optimality of the Rayleigh quotient
that $\rho$ is more accurate than $|\vartheta|$ in the sense that
$\|r(\rho,\widetilde u,\widetilde  v)\|\leq \|r(|\vartheta|,\widetilde u,
\widetilde v)\|$. Notice that in computation, by \eqref{equation:30}
we will take
$\widetilde u=U\frac{\widetilde c}{\|\widetilde c\|},\
\widetilde v=V\frac{\widetilde d}{\|\widetilde d\|}$
for $\widetilde c^T H\widetilde d\geq 0$ and
$\widetilde u=U\frac{\widetilde c}{\|\widetilde c\|},\
\widetilde v=-V\frac{\widetilde d}{\|\widetilde d\|}$
for $\widetilde c^T H\widetilde d<0$. This guarantees that
$\rho=\widetilde u^TA\widetilde v \geq 0$ and that $\widetilde u$ and
$\widetilde v$ are correct approximate left and right singular vectors.

\begin{algorithm}\label{algorithm:1}
\caption{HJDSVD and RHJDSVD methods with the fixed target $\tau$}
\begin{algorithmic}[1]
\STATE{\textbf{Input:\ }Devices to compute $A\underline v$ and
$A^T\underline u$ for arbitrary vectors $\underline u$ and
$\underline v$, unit-length starting vectors $u_0$ and $v_0$,
the target $\tau$ and convergence tolerance $tol$.}

\STATE{\textbf{Output:\ } A converged approximation $(\theta,  u ,  v )$
 to the desired
 $(\sigma, u^{*},v^{*})$ with $\sigma$ closest to $\tau$ satisfying
 $$
\|r\|= \left\|\begin{bmatrix}\begin{smallmatrix}A  v-
\theta   u\\ A^T  u-\theta   v\end{smallmatrix}\end{bmatrix}
\right\|\leq \|A\|_1\cdot tol.
 $$}

\STATE{Initialization:\ $u_{+}=u_0$, $v_{+}=v_0$; $U =[\ ]$, $V =[\ ]$.}

\FOR{$m=1,2,\ldots$}

%\STATE{Set $U =[U, u_{+}]$, $V=[V,v_{+}]$,
% and update $H =U^TAV$, $F=\begin{bmatrix}\begin{smallmatrix}-\tau I_m&H\\H^T&-\tau I_m
% \end{smallmatrix}\end{bmatrix}$
% and $G=\begin{bmatrix}\begin{smallmatrix}U^TAA^TU+
%\tau^2I_m&-2\tau H\\-2\tau H^T&V^TA^TAV+\tau^2I_m\end{smallmatrix}\end{bmatrix}$.}

\STATE{Set $U =[U, u_{+}]$, $V=[V,v_{+}]$,
 and update $H =U^TAV$, $G^{(1)}=U^TAA^TU$, $G^{(2)}=V^TA^TAV$, $F=\begin{bmatrix}\begin{smallmatrix}-\tau I_m&H\\H^T&-\tau I_m
 \end{smallmatrix}\end{bmatrix}$
 and $G=\begin{bmatrix}\begin{smallmatrix}G^{(1)}+
\tau^2I_m&-2\tau H\\-2\tau H^T&G^{(2)}+\tau^2I_m\end{smallmatrix}\end{bmatrix}$.}

\STATE{Compute the largest eigenvalue $\frac{1}{\nu}$ in magnitude of $(F,G)$,
its associated eigenvector $\widetilde f=\begin{bmatrix}\begin{smallmatrix}
\widetilde c^T,\ \widetilde d^T\end{smallmatrix}\end{bmatrix}^T$
and $\rho=\frac{|\widetilde c^TH\widetilde d|}{\|\widetilde c\|
\|\widetilde d\|}$.}

%\STATE{RHJDSVD: Form $G^{\prime}=\begin{bmatrix}\begin{smallmatrix}U^TAA^TU+\rho^2I_m&-2\rho H\\
%-2\rho H^T &V^TA^TAV+\rho^2I_m\end{smallmatrix}\end{bmatrix}$. Compute the eigenvector $\widehat f=\begin{bmatrix}\begin{smallmatrix}
%\widehat c^T,\widehat d^T\end{smallmatrix}\end{bmatrix}^T$ corresponding to the smallest eigenvalue of $G^{\prime}$ and $\rho^{\prime}=
%\tfrac{|\widehat c^TH\widehat d|}{\|\widehat c\|\|\widehat d\|}$.}

\STATE{RHJDSVD: Form $G^{\prime}=\begin{bmatrix}\begin{smallmatrix}G^{(1)}+\rho^2I_m&-2\rho H\\
-2\rho H^T &G^{(2)}+\rho^2I_m\end{smallmatrix}\end{bmatrix}$. Compute the eigenvector $\widehat f=\begin{bmatrix}\begin{smallmatrix}
\widehat c^T,\widehat d^T\end{smallmatrix}\end{bmatrix}^T$ corresponding to the smallest eigenvalue of $G^{\prime}$ and $\rho^{\prime}=
\tfrac{|\widehat c^TH\widehat d|}{\|\widehat c\|\|\widehat d\|}$.}

\STATE{Compute the approximate singular triplet
\small{$$(\theta,u,v)=\left\{\begin{aligned}
(\rho, \widetilde u=U\tfrac{\widetilde c}{\|\widetilde c\|},\widetilde v=
V\tfrac{\widetilde d}{\|\widetilde d\|}(\mbox{or} -
V\tfrac{\widetilde d}{\|\widetilde d\|})) \quad  \quad &\mbox{for HJDSVD,}\\
(\rho^{\prime}, \widehat u=U\tfrac{\widehat c}{\|\widehat c\|},
\widehat v=V\tfrac{\widehat d}{\|\widehat d\|}(\mbox{or}-V\tfrac{\widehat d}
{\|\widehat d\|})) \quad  \quad & \mbox{for RHJDSVD,}
\end{aligned}\right.$$}
as well as the residual $r=\begin{bmatrix}\begin{smallmatrix}A v-
 \theta u\\ A^T  u-\theta v
 \end{smallmatrix}\end{bmatrix}$.}

\STATE{\textbf{if} $\|r\|\leq \|A\|_1\cdot tol$ \textbf{then} return $(\theta,u,v)$ and stop.}

\STATE{Solve the correction equation
 \begin{displaymath}
 \begin{bmatrix}\begin{matrix}I_M-  u  u^T& \\ &
 I_N-  v  v^T\end{matrix}\end{bmatrix}
 \begin{bmatrix}\begin{matrix}-\tau I_M & A\\A^T& -
 \tau I_N\end{matrix}\end{bmatrix}
 \begin{bmatrix}\begin{matrix}I_M-  u  u^T& \\ &
 I_N-  v  v^T\end{matrix}\end{bmatrix}
 \begin{bmatrix}\begin{matrix}s\\t\end{matrix}\end{bmatrix}=-r,
 \end{displaymath}
 where $(s,t)\perp\perp(  u,  v)$.}
\STATE{Orthonormalize $s$ and $t$ against $U$ and $V$ to get $u_{+}$ and $v_{+}$.}
\ENDFOR
\end{algorithmic}
\end{algorithm}

\subsection{Refined harmonic Jacobi-Davidson SVD method}\label{subsection:RHJDSVD}

We have already obtained an approximate singular value $\rho$ of $A$
using the HJDSVD method. With $\rho$ available,
we can use the refined extraction to compute new and better approximations
to $u^{*}$ and $v^{*}$. Such approach is called the refined
harmonic extraction, which seeks a unit vector
$\widehat w=\begin{bmatrix}\begin{smallmatrix}U&\\
&V\end{smallmatrix}\end{bmatrix}\widehat f\in \mathbb{R}^{M+N}$ such that
\begin{equation}\label{equation:19}
\|C\widehat w-\rho \widehat w\|=\min_{w\in span\{U,V\}, \|w\|=1}\|Cw-\rho w\|,
\end{equation}
where $C$ is defined as \eqref{augment} and
$\mbox{span}\{U,V\}$ denotes the range of
$\begin{bmatrix}\begin{smallmatrix}U&\\&V\end{smallmatrix}\end{bmatrix}$.

For the approximate eigenvector $\widetilde w$ obtained by HJDSVD,
it has been proved by Jia \cite{jia2004some}
that $\|(C-\rho I)\widehat w\|<\|(C-\rho I)\widetilde w\|$, provided that
$(\rho, \widetilde w)$ is not an exact eigenpair of $C$; moreover,
if $(\rho, \widetilde w)$ is not an exact eigenpair of $C$ and
there is another harmonic Ritz value of $C$ close to $\rho$,
then it may occur that
\begin{displaymath}
\|(C-\rho I)\widehat w\|\ll\|(C-\rho I)\widetilde w\|,
\end{displaymath}
meaning that $\widehat w$ can be a much more accurate approximation to
$w^{*}=\frac{1}{\sqrt2}\begin{bmatrix}\begin{smallmatrix}
u^{*}\\v^{*}\end{smallmatrix}\end{bmatrix}$.

(\ref{equation:19}) is equivalent to
\begin{equation}\label{formulate}
\left\|\begin{bmatrix} -\rho U &AV\\A^TU&-\rho V \end{bmatrix}
\widehat f\right\|=\min_{f\in\mathbb{R}^{2m}, \|f\|=1}\left\|
\begin{bmatrix} -\rho U &AV\\A^TU&-\rho V \end{bmatrix}f\right\|.
\end{equation}
Therefore, $\widehat f$ is the right singular vector of $D=\begin{bmatrix}
\begin{smallmatrix}-\rho U &AV\\A^TU&-\rho V\end{smallmatrix}\end{bmatrix}$
corresponding to its smallest singular value, that is,
$\widehat f$ is an eigenvector of
\begin{equation}\label{gprime}
G^{\prime}=D^TD=\begin{bmatrix} U^TAA^TU+\rho^2I_m&-2\rho H\\
-2\rho H^T &V^TA^TAV+\rho^2I_m
 \end{bmatrix}
\end{equation}
associated with its smallest eigenvalue. Jia \cite{jia2006using} proves
that the smallest singular value and associated right singular
vector $\widehat f$ of $D$ can be computed accurately to the working precision by
applying the QR algorithm to the eigenvalue problem of
the cross-product matrix $G^{\prime}$,
provided that the smallest singular
value of $D$ is not close to its second smallest one. So instead
of computing the SVD of $D$, we work on $G^{\prime}$ and
compute its eigenvector $\widehat f$. In this way,
we reduce the cost substantially by solving a small $2m\times 2m$ matrix
eigenvalue problem of $G^{\prime}$ other than computing SVD of the $(M+N)\times 2m$
matrix $D$.

Let $\widehat f=\begin{bmatrix}\begin{smallmatrix}\widehat c^T,\
\widehat d^T\end{smallmatrix}
\end{bmatrix}^T$, where
$\widehat c,\ \widehat d\in\mathbb{R}^{m}$.
We take $\widehat u=U\frac{\widehat c}{\|\widehat c\|}$
and $\widehat v=V\frac{\widehat d}{\|\widehat d\|}$ (or $\widehat v=-
V\frac{\widehat d}{\|\widehat d\|}$ ) as new approximations
to $u^{*}$ and $v^{*}$, respectively.

If $(\rho,\widehat u,\widehat v)$
has not yet converged, we need to expand $\mathcal{U}$ and $\mathcal{V}$ to
$\mathcal{U}_{+}$ and $\mathcal{V}_{+}$
in the same manner as done in HJDSVD.
Since the consistency of the correction equation \eqref{equation:1},
where $u=\widehat u$ and $v=\widehat v$,
requires the double orthogonality
of the residual $r$, we take
the Rayleigh quotient
\begin{equation} \label{equation:31}
\theta=\rho^{\prime}=|\widehat u^TA\widehat v|=\frac{|\widehat c^TU^TAV\widehat d|}
{\|\widehat c\|\|\widehat d\|}=\frac{|\widehat c^TH\widehat d|}{\|\widehat c\|
\|\widehat d\|}
\end{equation}
and compute the residual $r=r(\rho^{\prime},\widehat u,\widehat{v})$.
The solution vectors $s$ and $t$ are then used to expand $\mathcal{U}$
and $\mathcal{V}$, respectively, and the resulting method
is called the refined harmonic Jacobi-Davidson SVD (RHJDSVD) method.

Similar to HJDSVD, we take $\widehat u=U\frac{\widehat c}{\|\widehat c\|}$
and $\widehat v=V\frac{\widehat d}{\|\widehat d\|}$ for
$\widehat c^TH\widehat d\geq 0$ and $\widehat u=U\frac{\widehat c}{\|\widehat c\|}$,
$\widehat v=-V\frac{\widehat d}{\|\widehat d\|}$ for
$\widehat c^TH\widehat d<0$, which guarantees that $\rho^{\prime}=
\widehat u^TA\widehat v\geq 0$ and that $\widehat u$ and
$\widehat v$ are correct approximate left and right singular vectors.

Algorithm \ref{algorithm:1} describes the HJDSVD and RHJDSVD methods,
where, and in the next section, to unify our presentation, we denote by $(\theta, u,v)$
the current approximation $(\rho,\widetilde u,\widetilde v)$ or
$(\rho^{\prime},\widehat u,\widehat v)$  to $(\sigma, u^{*}, v^{*})$.

\section{A convergence analysis}\label{section:3}

In this section, we make a convergence analysis on the inexact
HJDSVD and RHJDSVD methods.
We should stress that in the analysis of this section, we
have no restriction to the way that an approximate singular triplet
$(\theta,u,v)$ is extracted, which means that all our analysis of this section
is applicable to the standard and refined extraction based JDSVD methods.
Before proceeding, we must point out that
our analysis is asymptotic and a rigorous one-step one, i.e., it is a local other
than global analysis since we must assume that the current $(\theta,u,v)$ is
reasonably good but not yet converged, whose precise meaning will be clear later.

Making use of the fact that
$\begin{bmatrix}\begin{smallmatrix}I_M-uu^T& \\ &I_N-vv^T\end{smallmatrix}\end{bmatrix}
\begin{bmatrix}\begin{smallmatrix}s\\t\end{smallmatrix}\end{bmatrix}
=\begin{bmatrix}\begin{smallmatrix}s \\ t \end{smallmatrix}\end{bmatrix}$,
we transform (\ref{equation:1}) into
\begin{eqnarray}\label{equation:20}
\begin{bmatrix} -\tau I_M  & A\\A^T& -\tau I_N \end{bmatrix}
\begin{bmatrix} s\\t \end{bmatrix} &
=& \begin{bmatrix} uu^T&\\ & vv^T \end{bmatrix}
\begin{bmatrix} -\tau I_M  & A\\A^T&  -\tau I_N \end{bmatrix}
\begin{bmatrix} s\\t \end{bmatrix}
- \begin{bmatrix} Av-\theta u\\ A^Tu-\theta v \end{bmatrix} \nonumber\\
&=& \begin{bmatrix} (u^TAt - \tau u^Ts)u\\ (v^TA^Ts - \tau v^Tt)v \end{bmatrix} -
\begin{bmatrix} -\tau I_M & A\\ A^T& -\tau I_N \end{bmatrix}
\begin{bmatrix} u\\ v \end{bmatrix}
 + \begin{bmatrix} (\theta - \tau)u\\ (\theta - \tau)v \end{bmatrix}\nonumber\\
&=& - \begin{bmatrix} -\tau I_M  & A\\ A^T&  -\tau I_N
 \end{bmatrix}
\begin{bmatrix} u\\ v \end{bmatrix} +
\begin{bmatrix} (\theta - \tau + u^TAt)u\\
(\theta - \tau + v^TA^Ts)v
 \end{bmatrix}
\end{eqnarray}
with the last equation holding since $(s,t)\perp\perp (u,v$) means $u^Ts=v^Tt=0$.

Assume that $\tau>0$ is not a singular value of $A$. Since the eigenvalues of
$\begin{bmatrix}\begin{smallmatrix}-\tau I_M &A \\ A^T&-\tau I_N \end{smallmatrix}\end{bmatrix}$
 are $\pm\sigma_i-\tau,\ i=1,2,\ldots,N$ and $-\tau$,
$\begin{bmatrix}\begin{smallmatrix}-\tau I_M &A \\ A^T&-\tau I_N \end{smallmatrix}\end{bmatrix}$
is nonsingular.
Introduce the matrix
\begin{equation}\label{defB}
B=\begin{bmatrix} -\tau I_M & A\\A^T& -\tau I_N \end{bmatrix} ^{-1}.
\end{equation}
Then
\eqref{equation:20} means that the exact solution to (\ref{equation:1}) is
\begin{equation}\label{equation:3}
\begin{bmatrix} s\\t \end{bmatrix}
=-\begin{bmatrix} u\\v \end{bmatrix}+B
\begin{bmatrix} \alpha u\\ \beta v \end{bmatrix}
\quad\mbox{with} \quad
\left\{\begin{aligned}
&\alpha=\theta-\tau+u^TAt,\\
&\beta=\theta-\tau+v^TA^Ts.
\end{aligned}\right.
\end{equation}

%with
%\begin{equation}\label{equation:27}
%\left\{\begin{aligned}
%&\alpha=\theta-\tau+u^TAt,\\
%&\beta=\theta-\tau+v^TA^Ts.
%\end{aligned}\right.\end{equation}

Let $(\widetilde s,\widetilde t)$ be an approximate solution
of ($\ref{equation:1})$, and define its relative error by
\begin{equation}\label{equation:2}
 \varepsilon =\frac{
 \left\|\begin{bmatrix} \widetilde{s}\\ \widetilde {t} \end{bmatrix}-
 \begin{bmatrix}  s\\t   \end{bmatrix}\right\|}{\sqrt{\|s\|^2+\|t\|^2}}.
\end{equation}
Then we can write
\begin{displaymath}
 \begin{bmatrix} \widetilde{s}\\ \widetilde {t} \end{bmatrix}
 =\begin{bmatrix} s\\t \end{bmatrix}+\varepsilon \left\|
 \begin{bmatrix}  s\\t \end{bmatrix}\right\|
 \begin{bmatrix} g\\h \end{bmatrix},
 \end{displaymath}
 where $[g^T,h^T]^T$ is the normalized error direction vector with
 $\|g\|^2+\|h\|^2=1$. As a result, we obtain
 \begin{equation}\label{equation:5}
 \begin{bmatrix} (I_M-P_U)\widetilde{s}\\(I_N-P_V)\widetilde{t}\end{bmatrix}=
 \begin{bmatrix} (I_M-P_U) {s}\\(I_N-P_V){t} \end{bmatrix}+\varepsilon
 \left\|\begin{bmatrix} s\\t \end{bmatrix}\right\|
 \begin{bmatrix} g_{\perp}\\h_{\perp} \end{bmatrix},
 \end{equation}
 where $g_{\perp}=(I_M-P_U)g$ and $h_{\perp}=(I_N-P_V)h$. Define
 %\begin{displaymath}
% g_{\perp}=(I_M-P_U)g \quad \mbox{and} \quad h_{\perp}=(I_N-P_V)h.
% \end{displaymath}
% Define
 \begin{equation}\label{expansion}
 \widetilde{u}_{+}=\frac{(I_M-P_U)\widetilde s}{\|(I_M-P_U)\widetilde s\|}
 \quad  \mbox{and}  \quad
 \widetilde{v}_{+}=\frac{(I_N-P_V)\widetilde t}{\|(I_N-P_V)\widetilde t\|},
\end{equation}
 which are the normalized expansion vectors used to expand the current
 subspaces $\mathcal{U}$ and $\mathcal{V}$, respectively.

We measure the differences between $(I_M-P_U)\widetilde s$ and $(I_M-P_U)s$
and between $(I_N-P_V)\widetilde t$ and $(I_N-P_V)t$ by the relative errors
 \begin{eqnarray}
 \widetilde{\varepsilon}_s&=&\frac{\|(I_M-P_U)\widetilde s-(I_M-P_U)s\|}{\|(I_M-P_U)  s\|},
 \label{var1}\\
 \widetilde{\varepsilon}_t&=&\frac{\|(I_N-P_V)\widetilde t-(I_N-P_V)t\|}{\|(I_N-P_V)  t\|},
 \label{var2}
 \end{eqnarray}
 %\begin{equation}\widetilde{\varepsilon}_s=\frac{\|(I_M-P_U)\widetilde s-(I_M-P_U)s\|}{\|(I_M-P_U)  s\|}\quad \mbox{and} \quad
% \widetilde{\varepsilon}_t=\frac{\|(I_N-P_V)\widetilde t-(I_N-P_V)t\|}{\|(I_N-P_V)  t\|}\end{equation}
 respectively, or, equivalently, by the sines $\sin \angle (\widetilde u_{+},u_{+})$
 and $\sin \angle (\widetilde v_{+},v_{+})$. As a matter of fact,
 from Lemma 1 of \cite{jia2014inner}, it holds that
 \begin{align*}
& \sin \angle (\widetilde u_{+},u_{+})=\widetilde{\varepsilon}_s
\sin\angle (\widetilde u_{+},g_{\perp}),\\
& \sin\angle (\widetilde v_{+},v_{+})=\widetilde{\varepsilon}_t
\sin\angle(\widetilde v_{+},h_{\perp}).
\end{align*}
Note that $\sin \angle (\widetilde u_{+},g_{\perp})$ and
$\sin\angle (\widetilde v_{+},h_{\perp})$ are generally modest since
$\widetilde u_{+}$ and $\widetilde v_{+}$ lie
in the orthogonal complements of $\mathcal{U}$ and $\mathcal{V}$, respectively,
and $g_{\perp}$ and $h_{\perp}$ are general vectors in the corresponding
orthogonal complements. The above two relations show that
$\sin \angle (\widetilde u_{+},u_{+})$ and
$\sin\angle (\widetilde v_{+},v_{+})$ play the same role
as $\widetilde{\varepsilon}_s$ and $\widetilde{\varepsilon}_t$ when measuring
the differences between the inexact and exact expansion vectors.

Denote $\widetilde U_{+}=[U, u_{+}]$ and $\widetilde V_{+}=[V, v_{+}]$,
whose columns form orthonormal bases of
$\widetilde{\mathcal{U}}_{+}=span\{\widetilde U_{+}\}$ and
$\widetilde{\mathcal{V}}_{+}=span\{\widetilde V_{+}\}$, respectively.
In order to make the inexact JDSVD methods mimic the exact JDSVD methods
well, a {\em necessary} condition is that the two pairs of expanded subspaces
$\widetilde{\mathcal{U}}_{+}$, $\widetilde{\mathcal{V}}_{+}$ and
$\mathcal{U}_{+}$, $\mathcal{V}_{+}$ have very comparable qualities,
namely, $\sin\angle(\widetilde {\mathcal{U}}_{+},u^{*})\approx
\sin\angle( {\mathcal{U}}_{+},u^{*})$ and $\sin\angle
(\widetilde{\mathcal{V}}_{+},v^{*})\approx\sin\angle({\mathcal{V}}_{+},v^{*})$.
This condition is also {\em sufficient} for the inexact and exact RHJDSVD, though
it is not for the inexact and exact HJDSVD \cite{jia2015harmonic,wu17}. We
will describe the precise meaning of '$\approx$' afterwards. The following lemma
is adapted from Theorem 4.1 of \cite{jia2014inner}.

\begin{lemma}\label{lemma1}
Assume that $\sin\angle(u_{+},u^{*}_{\perp})\not=0$ and
$\sin\angle(v_{+},v^{*}_{\perp})\not=0$
with $u^{*}_{\perp}=(I_M-P_{U})u^{*}$ and
$v^{*}_{\perp}=(I_N-P_{V})v^{*}$\footnote{If either condition fails to
hold, it is seen from (\ref{equation:32}) that
$\sin\angle(\mathcal{U}_{+},u^{*})=0$
or $\sin\angle(\mathcal{V}_{+},v^{*})=0$. In this case, $\mathcal{U}_{+}$
contains $u^{*}$ or $\mathcal{V}_{+}$ contains $v^{*}$, that is,
$\mathcal{U}_{+}$ or $\mathcal{V}_{+}$
is already the best, and the expansion of
the other one subspace still satisfies
the lemma.}, and let $\widetilde{\varepsilon}_s$ and
$\widetilde{\varepsilon}_t$ be defined by \eqref{var1} and
\eqref{var2}. Then we have
\begin{equation}\label{equation:32}\begin{aligned}
&\sin\angle(\mathcal{U}_{+},u^{*})=\sin\angle(\mathcal{U},u^{*})
\sin\angle(u_{+},u^{*}_{\perp}),\\
&\sin\angle(\mathcal{V}_{+},v^{*})=\sin\angle(\mathcal{V},v^{*})
\sin\angle(v_{+},v^{*}_{\perp}),
\end{aligned}
\end{equation}
and
\begin{equation}\label{relation}
\frac{\sin\angle(\widetilde{\mathcal{U}}_{+},u^{*})}
{\sin\angle(\mathcal{U}_{+},u^{*})} =\frac{\sin
\angle(\widetilde u_{+},u^{*}_{\perp})}{\sin\angle(u_{+},u^{*}_{\perp})},  \quad  \quad
\frac{\sin\angle(\widetilde{\mathcal{V}}_{+},v^{*})}
{\sin\angle(\mathcal{V}_{+},v^{*})} =\frac{\sin\angle
(\widetilde v_{+},v^{*}_{\perp})}{\sin\angle(v_{+},v^{*}_{\perp})}.
\end{equation}
If $\tau_s=\frac{2\widetilde{\varepsilon}_s}{\sin\angle(u_{+},u^{*}_{\perp})}<1$ and
$\tau_t=\frac{2\widetilde{\varepsilon}_t}{\sin\angle(v_{+},v^{*}_{\perp})}<1$, we have
\begin{equation}\label{equation:33}\begin{aligned}
&1-\tau_s\leq\frac{\sin\angle(\widetilde{\mathcal{U}}_{+},u^{*})}
{\sin\angle(\mathcal{U}_{+},u^{*})}\leq 1+\tau_s,\\
&1-\tau_t\leq\frac{\sin\angle(\widetilde{\mathcal{V}}_{+},v^{*})}
{\sin\angle(\mathcal{V}_{+},v^{*})}\leq 1+\tau_t.
\end{aligned}\end{equation}
\end{lemma}

(\ref{equation:32}) indicates that
$\sin\angle(u_{+},u^{*}_{\perp})$ and $\sin\angle(v_{+},v^{*}_{\perp})$
are precisely the decreasing factors of one step subspace improvements
when $\mathcal{U}$ and $\mathcal{V}$ are expanded to
$\mathcal{U}_{+}$ and $\mathcal{V}_{+}$, respectively.
\eqref{relation} establishes the accuracy relationships between
the exact and inexact expanded subspaces.
(\ref{equation:33}) shows that, in order to make $\sin\angle
(\widetilde{\mathcal{U}}_{+},u^{*})\approx\sin
\angle( {\mathcal{U}}_{+},u^{*})$ and $\sin\angle
(\widetilde{\mathcal{V}}_{+},v^{*})\approx\sin\angle({\mathcal{V}}_{+},v^{*})$,
$\tau_s$ and $\tau_t$ should be considerably smaller than one.
Clearly, fairly small $\tau_s$
and $\tau_t$, e.g., $\tau_s,\tau_t=0.001$ or $0.01$, are enough
since we will have
\begin{equation}\label{equation:48}
\frac{\sin\angle(\widetilde{\mathcal{U}}_{+},u^{*})}
{\sin\angle(\mathcal{U}_{+},u^{*})},\
\frac{\sin\angle(\widetilde{\mathcal{V}}_{+},v^{*})}
{\sin\angle(\mathcal{V}_{+},v^{*})}\in[0.999,1.001]  \ \mbox{or} \ [0.99,1.01]
\end{equation}
and the differences between the lower and upper bounds are marginal,
which means that
$\widetilde{\mathcal{U}}_{+}$ and $\widetilde{\mathcal{V}}_{+}$ have
essentially the same quality as $\mathcal{U}_{+}$ and $\mathcal{V}_{+}$.
In the sense of \eqref{equation:48}, we claim that
 $\sin\angle(\widetilde {\mathcal{U}}_{+},u^{*})\approx
\sin\angle( {\mathcal{U}}_{+},u^{*})$ and $\sin\angle
(\widetilde{\mathcal{V}}_{+},v^{*})\approx\sin\angle({\mathcal{V}}_{+},v^{*})$.

From the definition of $\tau_s$ and $\tau_t$, we have
\begin{equation}\label{equation:49}
\widetilde{\varepsilon}_s=\frac{\tau_s}{2}\sin\angle(u_{+},u^{*}_{\perp}), \quad  \quad
\widetilde{\varepsilon}_t=\frac{\tau_t}{2}\sin\angle(v_{+},v^{*}_{\perp}).
\end{equation}
Notice that $\sin\angle(u_{+},u^{*}_{\perp})$ and $\sin\angle(v_{+},v^{*}_{\perp})$
are a-prior quantities and unknown during computation. For symmetric matrices,
the analysis of \cite{jia2014inner} has shown that the sizes of
$\sin\angle(u_{+},u^{*}_{\perp})$ and $\sin\angle(v_{+},v^{*}_{\perp})$
uniquely depend on the eigenvalue distribution of $B$, and the gap of the
desired singular value $\sigma^*$ and the other singular values of $A$:
the better $\sigma^*$ is separated from the others, the smaller these
two quantities, so that, by (\ref{equation:32}),
the more effectively the subspaces are expanded and the faster the JDSVD type
methods converge.

Generally speaking, we should not expect
that a practical SVD problem is too well conditioned,
that is, $\sin\angle(u_{+},u^{*}_{\perp})$ and $\sin\angle(v_{+},v^{*}_{\perp})$
are not very small. In applications we may well assume
that $\sin\angle(u_{+},u^{*}_{\perp})$ and $\sin\angle(v_{+},v^{*}_{\perp})$
are no smaller than $0.2$. We should be aware that the value 0.2 means
that $\sigma^*$ is quite well separated with the other singular values of $A$.
Indeed, it is instructive to see that $0.2^{15}\approx 3.3\times 10^{-11}$,
which is small enough and, by Lemma~\ref{lemma1}, means
that after {\em only} 15 outer iterations the subspaces $\mathcal{U}$ and $\mathcal{V}$
already contain sufficiently accurate approximations to the
desired singular vectors $u^*$ and $v^*$ since, at this time,
$\sin\angle (\mathcal{U},u^*)\leq 0.2^{15}\times \sin\angle (u_0,u^*)$
and $\sin\angle (\mathcal{V},v^*)\leq 0.2^{15}\times \sin\angle (v_0,v^*)$,
with $u_0$ and $v_0$ being starting vectors, are small enough. As we have elaborated,
fairly small $\tau_s,\tau_t\in [10^{-3},10^{-2}]$ will make the inexact and
exact expanded subspaces have the same quality, which should generally
make the inexact JDSVD methods behave like their exact counterparts, especially
for RHJDSVD, independent of the separation of $\sigma^*$ from the others.
As a result, suppose
that $\sin\angle(u_{+},u^{*}_{\perp})$ and $\sin\angle(v_{+},v^{*}_{\perp})$
are no smaller than $0.2$. In order to make the inexact
JDSVD methods mimic their exact counterparts well,
by \eqref{equation:48} and \eqref{equation:49} it is enough to take
\begin{equation}\label{tildevar}
\widetilde{\varepsilon}_s,\widetilde{\varepsilon}_t\in [10^{-4},10^{-3}]
\end{equation}
when $\sigma^*$ is well separated from the other singular values of $A$.

On the other hand, a very important point is that if
$\sin\angle(u_{+},u^{*}_{\perp})$  and $\sin\angle(v_{+},v^{*}_{\perp})$
are not small, that is, the SVD problem is not so well conditioned,
then JDSVD type methods may need many outer iterations to generate
sufficiently accurate subspaces
$\mathcal{U}$ and $\mathcal{V}$. In this case, in order to make the inexact
JDSVD methods mimic their exact counterparts well, by \eqref{equation:48}
and \eqref{equation:49}
we can {\em relax} $\widetilde{\varepsilon}_s$ and $\widetilde{\varepsilon}_t$
and take them {\em bigger} than those in \eqref{tildevar}, that is,
the more poorly is $\sigma^*$ separated from the others, the more easily
the inexact JDSVD mimics the exact JDSVD since we are allowed to take bigger
$\widetilde{\varepsilon}_s$ and $\widetilde{\varepsilon}_t$ than those in
\eqref{tildevar}, that is, we are allowed that $\widetilde{s}$ and $\widetilde{t}$
have poorer accuracy as approximations to $s$ and $t$.
In the sequel, we denote
\begin{equation}\label{widevarepsion}
\widetilde{\varepsilon}=\max\{\widetilde{\varepsilon}_s,\widetilde{\varepsilon}_t\}.
\end{equation}
The above analysis shows that
\begin{equation}\label{tildevarsize}
\widetilde{\varepsilon}\in [10^{-4}, 10^{-3}]
\end{equation}
is generally robust and reliable for the inexact JDSVD to mimic the exact JDSVD.

Our next goal is to derive relationships between the accuracy requirement
$\varepsilon$,  defined by (\ref{equation:2}), of
the approximate solution of the correction equation \eqref{equation:1}
and the accuracy $\widetilde\varepsilon$ of the expansion vectors
$\widetilde{s}$ and $\widetilde{t}$. Define
 \begin{equation}\label{hatvar}
 \widehat{\varepsilon}=\frac{\left\|\begin{bmatrix}\begin{smallmatrix}(I_M-P_U)
 \widetilde s\\ (I_N-P_V)\widetilde t\end{smallmatrix}\end{bmatrix}-
 \begin{bmatrix}\begin{smallmatrix}(I_M-P_U)  s\\ (I_N-P_V)  t
 \end{smallmatrix}\end{bmatrix}\right\|}{\left\|\begin{bmatrix}
 \begin{smallmatrix}(I_M-P_U)  s\\ (I_N-P_V)  t\end{smallmatrix}\end{bmatrix}\right\|},
 \end{equation}
which is the relative error of the {\em whole} inexact expansion vector
$[((I_M-P_U)\widetilde s)^T, ((I_N-P_V)\widetilde t)^T]^T$  and
the exact $[((I_M-P_U)s)^T,((I_N-P_V)t)^T]^T$. We have the following result.

\begin{lemma}\label{lemma2}
With $\widetilde{\varepsilon}$ and $\widehat{\varepsilon}$ defined by
\eqref{widevarepsion} and \eqref{hatvar}, it holds that
\begin{equation}\label{equation:35}
\widehat{\varepsilon}\leq \widetilde{\varepsilon}.
\end{equation}
\end{lemma}
\begin{proof}
From \eqref{var1},
\eqref{var2} and \eqref{widevarepsion} we obtain
 \begin{align*}
 \widehat{\varepsilon}
 &=\sqrt{\frac{\|(I_M-P_U)(\widetilde s- s)\|^2+  \|(I_N-P_V)
 (\widetilde t-t)\|^2}{\|(I_M-P_U)s\|^2+ \|(I_N-P_V)t\|^2}}\\
 &=\sqrt{\frac{\widetilde\varepsilon_s^2\|(I_M-P_U)s\|^2+
 \widetilde\varepsilon_t^2\|(I_N-P_V)t\|^2}{\|(I_M-P_U)s\|^2+ \|(I_N-P_V)t\|^2}}\\
 &\leq\widetilde\varepsilon.
 \end{align*}
\end{proof}

With the help of this lemma, we can establish the first result on the relationship
between the solution accuracy $\varepsilon$ %defined by \eqref{equation:2}
of correction equation \eqref{equation:1} and
$\widetilde{\varepsilon}$. %defined by \eqref{widevarepsion}.

\begin{theorem}\label{theorem:2}
 Let $(u,v)$ be the current approximation to $(u^{*},v^{*})$,
 $\varepsilon$ and $\widetilde{\varepsilon}$ be defined by \eqref{equation:2}
 and \eqref{widevarepsion}, and $\alpha$ and $\beta$ defined by
 \eqref{equation:3}. Then
 \begin{equation}\label{equation:25}
 \varepsilon\leq\frac{2\sqrt{\alpha^2+\beta^2}\sin_{\max}}
 {|\sigma-\tau|\left\|B\begin{bmatrix}
 \begin{smallmatrix} \alpha u\\ \beta v\end{smallmatrix}\end{bmatrix}-
 \begin{bmatrix}\begin{smallmatrix}u\\
 v\end{smallmatrix}\end{bmatrix}\right\|\sqrt{\|g_{\perp}\|^2+\|h_{\perp}\|^2}}
 \widetilde{\varepsilon},
 \end{equation}
 where
 \begin{equation}\label{sinmax}
 \sin_{\max}=\max\{\sin\varphi, \sin\psi\}
  \end{equation}
 with the acute angles
 \begin{equation*}
 \varphi=\angle(u,u^{*}), \quad  \quad  \psi=\angle(v,v^{*}).
 \end{equation*}
 \end{theorem}

 \begin{proof}
 From (\ref{equation:5}) and \eqref{hatvar}, we obtain
 \begin{displaymath}
 \varepsilon=\frac{\left\|\begin{bmatrix}\begin{smallmatrix}
 (I_M-P_U)\widetilde{s}\\(I_N-P_V)\widetilde{t}\end{smallmatrix}
 \end{bmatrix}-\begin{bmatrix}\begin{smallmatrix}(I_M-P_U) {s}\\
 (I_N-P_V){t}\end{smallmatrix}\end{bmatrix}\right\|}
 {\left\|\begin{bmatrix}\begin{smallmatrix}s\\t\end{smallmatrix}
 \end{bmatrix}\right\|\sqrt{\|g_{\perp}\|^2+\|h_{\perp}\|^2}}=
 \frac{\left\|\begin{bmatrix}\begin{smallmatrix} (I_M-P_U)& \\ &(I_N-P_V)
 \end{smallmatrix}\end{bmatrix}\begin{bmatrix}\begin{smallmatrix} s\\
  t\end{smallmatrix}\end{bmatrix}\right\|}{\left\|\begin{bmatrix}
  \begin{smallmatrix}s\\ t\end{smallmatrix}\end{bmatrix}
  \right\|\sqrt{\|g_{\perp}\|^2+\|h_{\perp}\|^2}}
 \widehat{\varepsilon}.
 \end{displaymath}
By (\ref{equation:3}), substituting $\begin{bmatrix}\begin{smallmatrix}s\\t
\end{smallmatrix}\end{bmatrix}= -\begin{bmatrix}\begin{smallmatrix}u\\v
\end{smallmatrix}\end{bmatrix}+ B\begin{bmatrix}\begin{smallmatrix}\alpha u\\
\beta v\end{smallmatrix}\end{bmatrix}$ into the above and
making use of $u\in\mathcal{U}$ and $v\in\mathcal{V}$, we obtain $(I_M-P_U)u=0,\
(I_N-P_V)v=0$ and
\begin{equation}\label{equation:24}
\begin{aligned}
 \varepsilon &=\frac{\left\|\begin{bmatrix}\begin{smallmatrix} (I_M-P_U)& \\
 &(I_N-P_V)\end{smallmatrix}\end{bmatrix}B\begin{bmatrix}\begin{smallmatrix}\alpha u\\
 \beta v\end{smallmatrix}\end{bmatrix}\right\|}{\left\|B
 \begin{bmatrix}\begin{smallmatrix}\alpha u\\ \beta v
 \end{smallmatrix}\end{bmatrix}-\begin{bmatrix}
 \begin{smallmatrix}u\\v\end{smallmatrix}\end{bmatrix}
 \right\|\sqrt{\|g_{\perp}\|^2+ \|h_{\perp}\|^2}}
 \widehat{\varepsilon}.
 \end{aligned}
 \end{equation}

 Decompose $u$ and $v$ into the orthogonal direct sums
 \begin{equation}\label{equation:10}
 \left\{\begin{aligned}
 &u=u^{*}\cos\varphi+p\sin\varphi\\
 &v=v^{*}\cos\psi+q\sin\psi
 \end{aligned}\right.
 \end{equation}
 with $p\perp u^{*},\ q\perp v^{*}$ and $\|p\|=\| q\|=1$.

For $\sigma\neq\tau$, there exist unique $\alpha^{\prime}$ and $\beta^{\prime}$
that satisfy the $2\times 2$ linear system
 \begin{equation}\label{dimension2}
 \begin{bmatrix} -\tau&\sigma\\ \sigma&-\tau \end{bmatrix}
 \begin{bmatrix} \alpha^{\prime}\\ \beta^{\prime} \end{bmatrix}=
 \begin{bmatrix} \alpha \cos\varphi \\ \beta\cos\psi \end{bmatrix}, \quad
 \mbox{i.e.,} \quad
 \begin{bmatrix}  \alpha^{\prime}\\ \beta^{\prime}  \end{bmatrix}=
 \begin{bmatrix} -\tau&\sigma\\ \sigma&-\tau \end{bmatrix}^{-1}
 \begin{bmatrix} \alpha \cos\varphi \\ \beta\cos\psi \end{bmatrix}.
 \end{equation}

Since $(\sigma,u^*,v^*)$ is a singular triplet of $A$, it follows
from \eqref{dimension2} and the  definition \eqref{defB} of $B$ that
 \begin{displaymath}
 \begin{bmatrix} -\tau I_M&A\\A^T &   -\tau I_N
  \end{bmatrix}
 \begin{bmatrix} \alpha^{\prime}u^{*}\\
  \beta^{\prime}v^{*} \end{bmatrix}=
  \begin{bmatrix} \alpha\cos\varphi u^{*}\\
 \beta\cos\psi v^{*} \end{bmatrix}, \quad \mbox{i.e.,} \quad
 \begin{bmatrix} \alpha^{\prime} u^{*}\\ \beta^{\prime}v^{*}
 \end{bmatrix}=B\begin{bmatrix} \alpha \cos\varphi u^{*}\\ \beta\cos\psi v^{*}
  \end{bmatrix}.
 \end{displaymath}
Therefore, from \eqref{equation:10} and the above relation we obtain
 \begin{equation}\label{equation:22}
 B\begin{bmatrix} \alpha u\\ \beta v \end{bmatrix}=
 B\begin{bmatrix} \alpha \cos\varphi u^{*}\\
  \beta\cos\psi v^{*} \end{bmatrix}+
 B\begin{bmatrix} \alpha \sin\varphi p\\
 \beta\sin\psi q \end{bmatrix}=
 \begin{bmatrix} \alpha^{\prime} u^{*}\\
 \beta^{'}v^{*} \end{bmatrix}+B
 \begin{bmatrix} \alpha \sin\varphi p\\
 \beta\sin\psi q \end{bmatrix}.
 \end{equation}

 Taking norms on both sides of the second relation in
 \eqref{dimension2}, we obtain
 \begin{equation}\label{equation:50}
 \sqrt{(\alpha^{\prime})^2 + (\beta^{\prime})^2}
  \leq  \frac{1}{|\sigma -  \tau|}\sqrt{( \alpha\cos\varphi )^2
 + ( \beta\cos\psi )^2}
 \leq \frac{1}{|\sigma-\tau|}
 \sqrt{\alpha^2 + \beta^2}.
 \end{equation}

For the numerator of \eqref{equation:24},  exploiting (\ref{equation:22}),
we have
 \begin{align}
&\left\|\begin{bmatrix}  I_M - P_U &\\ & I_N - P_V \end{bmatrix}B
 \begin{bmatrix} \alpha u\\ \beta v \end{bmatrix} \right\| \nonumber \\
& \quad =\left\|\begin{bmatrix}  \alpha^{\prime}(I_M-P_U)u^{*}\\
 \beta^{\prime}(I_N-P_V)v^{*} \end{bmatrix}+
 \begin{bmatrix} I_M -  P_U &\\ & I_N - P_V \end{bmatrix}B
 \begin{bmatrix} \alpha \sin\varphi p\\ \beta
 \sin\psi q \end{bmatrix}\right\|  \nonumber \\
& \quad \leq \left\| \begin{bmatrix}\alpha^{\prime}(I_M-P_U)u^{*}\\
 \beta^{\prime}(I_N-P_V)v^{*} \end{bmatrix} \right\| +
 \left\| \begin{bmatrix} I_M - P_U &\\ & I_N - P_V
  \end{bmatrix} B \begin{bmatrix}
  \alpha \sin\varphi p\\ \beta \sin\psi q
  \end{bmatrix} \right\| \nonumber  \\
& \quad \leq\sqrt{(\alpha^{\prime})^2\|u^{*}_{\perp}\|^2 + (\beta^{\prime})^2
 \| v^{*}_{\perp}\|^2} + \frac{1}{|\sigma-\tau|}\sqrt{\alpha^2\sin^2\varphi +
 \beta^2\sin^2\psi}, \label{eqtm3:1}
 \end{align}
 where we have used
 $\|B\|=\frac{1}{|\sigma-\tau|}$ in the last inequality and
 $u^{*}_{\perp}=(I_M-P_U)u^{*}$, $v^{*}_{\perp}=(I_N-P_V)v^{*}$,
 as defined in Lemma~\ref{lemma1}. Notice from \eqref{sinmax} that
 \begin{align*}
 &\|u^{*}_{\perp}\|=\sin\angle(\mathcal{U},u^{*})\leq\sin\angle(u,u^{*})
 =\sin\varphi\leq  \sin_{\max},\\
 &\|v^{*}_{\perp}\|=\sin\angle(\mathcal{V},v^{*})\leq\sin\angle(v,v^{*})=
 \sin\psi\leq \sin_{\max}.
 \end{align*}
 Therefore, from the above relations and \eqref{equation:50} we obtain
 %\begin{subequations}
 \begin{align}
 \left\| \begin{bmatrix}I_M-P_U &\\ & I_N-P_V
 \end{bmatrix} B \begin{bmatrix} \alpha u\\
 \beta v \end{bmatrix}  \right\|
  &\leq  \left(\sqrt{(\alpha^{\prime})^2 + (\beta^{\prime})^2
 } + \frac{1}{|\sigma-\tau|}\sqrt{\alpha^2 + \beta^2}\right) \sin_{\max} \nonumber \\
 %&\leq  \left(\frac{1}{|\sigma - \tau|} + \|B\|\right)
% \sqrt{\alpha^2 + \beta^2}\sin_{\max}  \nonumber\\
 &=\frac{2}{|\sigma - \tau|}\sqrt{\alpha^2 + \beta^2}\sin_{\max}\label{eqtm3:2}
 \end{align}
 %\end{subequations}
 Then from (\ref{equation:35}) and (\ref{equation:24}) it follows
 that (\ref{equation:25}) holds.
 \end{proof}

In the following we analyze Theorem~\ref{theorem:2}
and obtain a more compact and insightful form by estimating $\alpha$ and $\beta$
accurately. To this end,
assume that the current $u$ and $v$ are reasonably good approximations to the
desired $u^{*}$ and $v^{*}$ with the same order accuracy, that is,
\begin{equation}\label{assum1}
\sin\varphi=\sin\angle(u,u^{*})=\mathcal{O}(\epsilon),  \quad  \quad
\sin\psi=\sin\angle(v,v^{*})=\mathcal{O}( \epsilon)
\end{equation}
with $\epsilon$ reasonably small, which means that
the exact solutions $[s^T,t^T]^T$
of the correction equations in Algorithm~\ref{algorithm:1} satisfy
\begin{equation}\label{assum2}
\|s\|=\mathcal{O}(\epsilon), \quad  \quad \|t\|=\mathcal{O}(\epsilon).
\end{equation}
Define the quantity
\begin{equation}\label{gamma}
\gamma=sign(\theta-\tau)\frac{\sqrt2}{\sqrt{\alpha^2+\beta^2}}
\end{equation}
with $sign(\cdot)$ being the sign function.

In what follows we estimate $\alpha,\ \beta$ and $\gamma$ accurately.

\begin{theorem}\label{alphabeta}
For $\alpha,\beta$ and $\gamma$ defined by \eqref{equation:3} and \eqref{gamma},
we have
\begin{align}
\alpha&=\theta-\tau+\mathcal{O}(\epsilon^2),\label{alpha}\\
\beta&=\theta-\tau+
\mathcal{O}(\epsilon^2),\label{beta}\\
\gamma&=\frac{1}{\theta-\tau}+\mathcal{O}(\epsilon^2).\label{gammaeq}
\end{align}
\end{theorem}

 \begin{proof}
From \eqref{equation:3} and \eqref{equation:10}, we have
\begin{displaymath}\quad\quad\left\{\begin{aligned}
&\alpha=\theta-\tau+t^TA^T(u^{*}\cos\varphi+p\sin\varphi),\\
&\beta=\theta-\tau+s^TA(v^{*}\cos\psi+q\sin\psi),
\end{aligned}\right.
\end{displaymath}
i.e.,
\begin{displaymath}\quad\quad\left\{\begin{aligned}
&\alpha-(\theta-\tau)=t^TA^Tp\sin\varphi+\sigma t^Tv^{*}\cos\varphi,\\
&\beta-(\theta-\tau)=s^TAq\sin\psi+\sigma s^Tu^{*}\cos\psi.
\end{aligned}\right.\end{displaymath}

Keep in mind that $\varphi=\angle(u,u^{*}),\ \ \psi=\angle(v,v^{*})$. Similarly to
\eqref{equation:10}, we now decompose $u^{*}$ and $v^{*}$ into the orthogonal
direct sums
\begin{equation*}\left\{\begin{aligned}
&u^{*}=u\cos\varphi+p^{\prime}\sin\varphi,\\
&v^{*}=v\cos\psi+q^{\prime}\sin\psi
\end{aligned}\right.\end{equation*}
with $p^{\prime}\perp u, q^{\prime}\perp v$ and $\|p^{\prime}\|=\|q^{\prime}\|=1$.
Since $s\perp u$ and $t\perp v$, we have
\begin{displaymath}\left\{\begin{aligned}
&\alpha-(\theta-\tau)=t^TA^Tp\sin\varphi+\sigma t^Tq^{\prime}\cos\varphi\sin\psi,\\
&\beta-(\theta-\tau)=s^TAq\sin\psi+\sigma s^Tp^{\prime}\cos\psi\sin\varphi.
\end{aligned}\right.
\end{displaymath}
Therefore, from \eqref{sinmax}, \eqref{assum1} and \eqref{assum2} we obtain
$$
\left\{\begin{aligned}
&|\alpha-(\theta-\tau)|\leq (\|A\|+\sigma)\|t\|\sin_{\max}=\mathcal{O}(\epsilon^2),\\
&|\beta-(\theta-\tau)|\leq(\|A\|+\sigma)\|s\|\sin_{\max}=\mathcal{O}(\epsilon^2).
\end{aligned}\right.
$$
Consequently, we obtain
$$
\alpha=\theta-\tau+\mathcal{O}(\epsilon^2), \quad  \quad \beta=\theta-\tau+
\mathcal{O}(\epsilon^2),
$$
which proves \eqref{alpha}, \eqref{beta}
and
\begin{equation}\label{alphabetabound}
|\alpha-\beta|\leq (\|A\|+\sigma)(\|t\|+\|s\|)\sin_{\max}
\leq 2\|A\|(\|t\|+\|s\|)\sin_{\max}=\mathcal{O}(\epsilon^2).
\end{equation}
From \eqref{alpha} and \eqref{beta} it is seen that
$\alpha$ and $\beta$ have the same signs as $\theta-\tau$
for $\epsilon$ reasonably small.
Moreover, $\alpha,\beta$ and $\gamma$ have the same signs, and
\begin{equation*}
\alpha\gamma=1+\mathcal{O}(\epsilon^2)  \quad \mbox{and} \quad
\beta\gamma=1+\mathcal{O}(\epsilon^2),
\end{equation*}
which yields \eqref{gammaeq} by combining \eqref{gamma} with
\eqref{alpha} or \eqref{beta}.
\end{proof}

By simple justification, we find from this theorem that
\begin{subequations}\label{equations}\begin{align*}
 \frac{1}{\sqrt{\alpha^2 + \beta^2}}\left\|B
 \begin{bmatrix}  \alpha u\\ \beta v \end{bmatrix} -
 \begin{bmatrix} u\\v\end{bmatrix} \right\|
 &= \frac{1}{\sqrt{2}}\left\| B
 \begin{bmatrix} u\\v \end{bmatrix} - \gamma
 \begin{bmatrix} u\\v \end{bmatrix}
 \right\| + \mathcal{O}(\epsilon^2) \label{equation:34}\\
  &= \frac{1}{\sqrt{2}}\left\| B
  \begin{bmatrix}u\\v \end{bmatrix} - \frac{1}{\theta - \tau}
 \begin{bmatrix} u\\v \end{bmatrix}
 \right\| + \mathcal{O}(\epsilon^2). \end{align*}
\end{subequations}
Therefore, within $\mathcal{O}(\epsilon^2)$,
the above left-hand side, which is a factor of the right-hand side
in \eqref{equation:25}, is the residual
norm of regarding $\gamma$ or $\frac{1}{\theta-\tau}$ as an
approximation to the eigenvalue $\frac{1}{\sigma-\tau}$
of $B$ and $\frac{1}{\sqrt{2}}
\begin{bmatrix}\begin{smallmatrix}u\\v
 \end{smallmatrix}\end{bmatrix}$ as the corresponding
normalized approximate eigenvector.

With the help of Theorem~\ref{alphabeta}, we are able to make Theorem~\ref{theorem:2}
clearer, more compact and insightful. Before doing so,
we need the following lemma, which is direct from
Theorem 6.1 of \cite{jia2001analysis}.

\begin{lemma}\label{lemma:1}
Let $\left(\frac{1}{\sigma-\tau},w=\frac{1}{\sqrt{2}}
\begin{bmatrix}\begin{smallmatrix}u^*\\v^*
 \end{smallmatrix}\end{bmatrix}\right)$
 be a simple eigenpair of $B$ and
 $[w,W_{\perp}]$ be orthogonal. Then
 \begin{displaymath}
 \begin{bmatrix} w^T\\ W_{\perp}^T \end{bmatrix}
 B[w, W_{\perp}]
 =\begin{bmatrix} \frac{1}{\sigma-\tau}&0\\0&L \end{bmatrix},
 \end{displaymath}
 where $L=W_{\perp}^TBW_{\perp}$ is a symmetric matrix.
 Let $\left(\gamma,z=\frac{1}{\sqrt{2}} \begin{bmatrix}\begin{smallmatrix}u\\v
 \end{smallmatrix}\end{bmatrix}\right)$
 be an approximation to $(\frac{1}{\sigma-\tau},w)$. Assume that
 $\gamma$ is not an eigenvalue of $L$, and define
 \begin{displaymath}
 {\rm sep}(\gamma,L)=\|(L-\gamma I)^{-1}\|^{-1}>0.
 \end{displaymath}
 Then
 \begin{equation}\label{equation:6}
 \sin\angle(z,w)\leq\frac{\|Bz-\gamma z\|}{{\rm sep}(\gamma,L)}.
 \end{equation}
 \end{lemma}

Based on Theorem \ref{theorem:2}
and Lemma \ref{lemma:1}, we can establish the following close relationship
between $\varepsilon$ and $\widetilde\varepsilon$ defined
by \eqref{equation:2} and \eqref{widevarepsion}, respectively.

\begin{theorem}\label{theorem:3}
Assume that $\gamma=sign(\theta-\tau)\frac{\sqrt2}
{\sqrt{\alpha^2+\beta^2}}$ is an approximation to $\frac{1}{\sigma-\tau}$
and is not an eigenvalue of $L$, and
let $\varepsilon$ and $\widetilde{\varepsilon}$ be defined by \eqref{equation:2}
and \eqref{widevarepsion}. Then
\begin{equation}\label{equation:15}
\varepsilon\leq\frac{2\sqrt2\delta}{{\rm sep}(\gamma,L) |\sigma-\tau|
\sqrt{\|g_{\perp}\|^2+\|h_{\perp}\|^2}}\widetilde{\varepsilon},
\end{equation}
where
\begin{equation}\label{equation:26}
\delta=\frac{\left\|B\begin{bmatrix}\begin{smallmatrix}u\\v
 \end{smallmatrix}\end{bmatrix}-\gamma \begin{bmatrix}\begin{smallmatrix}u\\v
\end{smallmatrix}\end{bmatrix}\right\|}{\left\|B\begin{bmatrix}
\begin{smallmatrix}\alpha\gamma u\\ \beta\gamma v \end{smallmatrix}
\end{bmatrix}-\gamma \begin{bmatrix}\begin{smallmatrix}u\\v
\end{smallmatrix}\end{bmatrix}\right\|}.\end{equation}
\end{theorem}

\begin{proof}
Denote $\phi=\angle\left(\begin{bmatrix}\begin{smallmatrix}u\\v\end{smallmatrix}
\end{bmatrix},
 \begin{bmatrix}\begin{smallmatrix}u^{*}\\v^{*}\end{smallmatrix}\end{bmatrix}
 \right).
$
Then $\phi= \angle(z,w)$ by noting the definition of $z$ and $w$
in Lemma~\ref{lemma:1}. From \eqref{equation:6} we obtain
\begin{displaymath}
\sin\phi\leq \frac{\left\|B\begin{bmatrix}\begin{smallmatrix}u\\v\end{smallmatrix}
\end{bmatrix}-
\gamma\begin{bmatrix}\begin{smallmatrix}u\\v\end{smallmatrix}\end{bmatrix}\right\|}
{\sqrt{2}{\rm sep}(\gamma,L)}.
\end{displaymath}
From Theorem 2.3 of \cite{jia2003implicitly}, we have
\begin{displaymath}
\sin^2\varphi+\sin^2\psi\leq2\sin^2\phi.
\end{displaymath}
Hence
\begin{equation}\label{equation:8}
\sin_{\max}=\max\{\sin\varphi,\sin\psi\}\leq\sqrt2\sin\phi\leq\frac{
\left\|B\begin{bmatrix} \begin{smallmatrix}u\\v\end{smallmatrix}\end{bmatrix}-
\gamma\begin{bmatrix}
\begin{smallmatrix}u\\v\end{smallmatrix}\end{bmatrix}\right\|}{{\rm sep}(\gamma, L)}.
\end{equation}
Making use of this relation and $\delta$ defined by (\ref{equation:26}),
we obtain from \eqref{equation:25} that
\begin{align*}
\varepsilon&\leq\frac{2\sqrt2\delta\sin_{\max}}
{|\sigma-\tau|\left\|B\begin{bmatrix} \begin{smallmatrix}u\\v
\end{smallmatrix}\end{bmatrix}- \gamma\begin{bmatrix}
\begin{smallmatrix}u\\v\end{smallmatrix}\end{bmatrix}\right\|
\sqrt{\|g_{\perp}\|^2+\|h_{\perp}\|^2}}\widetilde{\varepsilon} \nonumber\\
&\leq\frac{2\sqrt2\delta}{{\rm sep}(\gamma,L)|\sigma-\tau|
\sqrt{\|g_{\perp}\|^2 +\|h_{\perp}\|^2}}\widetilde{\varepsilon}.\nonumber
\end{align*}\end{proof}

Theorem~\ref{theorem:3} shows that once $\widetilde \varepsilon$ is
given we can determine the {\em least} or {\em lowest} accuracy
requirement $\varepsilon$ for the correction equation (\ref{equation:1}).

We will prove that $\delta=1+\mathcal{O}(\epsilon)$ afterwards.
(\ref{equation:5}) indicates that $\sqrt{\|g_{\perp}\|^2+\|h_{\perp}\|^2}$ is modest
since $g$ and $h$ are general vectors satisfying $\|g\|^2+\|h\|^2=1$
and $\|g_{\perp}\|\leq \|g\|$, $\|h_{\perp}\|\leq \|h\|$.
If $\sqrt{\|g_{\perp}\|^2+\|h_{\perp}\|^2}$ is small, the solution
accuracy requirement $\varepsilon$ can be relaxed for a
fixed small $\widetilde{\varepsilon}$, that is, the
correction equation (\ref{equation:1}) is allowed to be solved with {\em lower}
accuracy.
So a small $\sqrt{\|g_{\perp}\|^2+\|h_{\perp}\|^2}$ is a {\em lucky} event.

If $\gamma$ is well separated from the other eigenvalues
of $B$ than $\frac{1}{\sigma-\tau}$, then we have
${\rm sep}(\gamma,L)\approx{\rm sep}(\frac{1}{\sigma-\tau},L)\leq 2\|B\|
=\frac{2}{|\sigma-\tau|}$, which is tight. In this case, noting
that $\sqrt{\|g_{\perp}\|^2+\|h_{\perp}\|^2}\leq 1$, we obtain
 \begin{displaymath}
 \varepsilon\leq\frac{2\sqrt2\delta} {{\rm sep}(\gamma,L)|\sigma-\tau|
 \sqrt{\|g_{\perp}\|^2+\|h_{\perp}\|^2}}
 \widetilde{\varepsilon}=\frac{\sqrt2} {\mathcal{O}(1)}
 \widetilde\varepsilon=\mathcal{O}(\widetilde\varepsilon).
 \end{displaymath}
On the other hand, if $\frac{1}{\sigma-\tau}$ is not well separated from the
other eigenvalues of $B$, then
${\rm sep}(\gamma,L)\approx{\rm sep}(\frac{1}{\sigma-\tau},L)$
is considerably smaller than $\|B\|=\frac{1}{|\sigma-\tau|}$,
which may make $\varepsilon$ considerably bigger than $\widetilde\varepsilon$,
meaning that we are allowed to solve (\ref{equation:1}) with {\em lower} accuracy.
Therefore, as far as solving correction equation \eqref{equation:1} is concerned,
the bad separation of $\frac{1}{\sigma-\tau}$ from the other
eigenvalues of $B$ is a {\em lucky} event.

Summarizing the above, we conclude that, in any event,
$\varepsilon=c\widetilde{\varepsilon}$ with a fairly modest constant
$c=\mathcal{O}(1)$ is a reliable and general-purpose choice.

Now, we analyze $\delta$ and estimate its size accurately.
Recall \eqref{gamma}, and denote
\begin{equation}x=B\begin{bmatrix} \alpha u\\ \beta v\end{bmatrix}-
\begin{bmatrix}  u\\  v \end{bmatrix}, \quad  \quad
y=\frac{1}{\gamma}\left( B  \begin{bmatrix} u\\ v \end{bmatrix}-
 \gamma\begin{bmatrix}  u\\   v \end{bmatrix} \right).
 \end{equation}
Then it is known from \eqref{equation:26} that
$\delta=\frac{\|y\|}{\|x\|}$ and
\begin{equation}\label{equation:11}
\left|1-\frac{1}{\delta}\right|=\frac{|\|y\|-\|x\||}{\|y\|}\leq\frac{\|y-x\|}{\|y\|}.
\end{equation}
When the right-hand side of (\ref{equation:11}) is smaller than one,
we will have
\begin{equation}\label{lowup}
\frac{1}{1+\frac{\|y-x\|}{\|y\|}}\leq\delta\leq\frac{1}{1-\frac{\|y-x\|}{\|y\|}}.
\end{equation}

We next estimate the right-hand side of (\ref{equation:11}) and prove that
it is $\mathcal{O}(\epsilon)$.

Write
$z=\frac{1}{\sqrt2}\begin{bmatrix}\begin{smallmatrix} u\\ v
\end{smallmatrix}\end{bmatrix}$ and
$w=\frac{1}{\sqrt2}\begin{bmatrix}\begin{smallmatrix} u^*\\ v^*
\end{smallmatrix}\end{bmatrix}$. Then $y=\frac{\sqrt2}{\gamma}(Bz-\gamma z)$,
and from \eqref{equation:6} we have
\begin{displaymath}
\|y\|=\frac{\sqrt2}{|\gamma|}\|Bz-\gamma z\|\geq\frac{\sqrt2}{|\gamma|}
\sin\angle(z,w){\rm sep}(\gamma,L).
\end{displaymath}
Notice from (\ref{equation:8}) that $\sqrt2 \sin\angle(z,w)\geq \sin_{\max}$.
Then the above relation gives
\begin{equation}\label{equation:13}
\|y\|\geq\frac{ {\rm sep}(\gamma, L)\sin_{\max}}{|\gamma|}.
\end{equation}

As for $\|y-x\|$, from \eqref{gamma} we obtain
\begin{subequations}\begin{align}
\|y-x\|&=\left\|B\begin{bmatrix} (\frac{1}{\gamma}-\alpha)u \nonumber\\
(\frac{1}{\gamma}-\beta)v \end{bmatrix}\right\|
\leq \|B\| \sqrt{(\frac{1}{\gamma}-\alpha)^2+(\frac{1}{\gamma}-\beta)^2}\nonumber\\
%&=\|B\|\sqrt{\frac{2}{\gamma^2}+\alpha^2+\beta^2-\frac{2\alpha+2\beta}{\gamma}}\\
&=\frac{1}{|\sigma-\tau|} \sqrt{2(\alpha^2+\beta^2)-\sqrt{2(\alpha^2+\beta^2)}{\rm sign}
(\theta-\tau)(\alpha+\beta)}\nonumber\\
&=\frac{1}{|\sigma-\tau|} \sqrt{ \sqrt{2(\alpha^2 + \beta^2)}[\sqrt{2(\alpha^2 + \beta^2)} - {\rm sign}
(\theta - \tau)(\alpha + \beta)]}\nonumber\\
&=\frac{1}{|\sigma-\tau|} \sqrt{\frac{\sqrt{2\alpha^2 + 2\beta^2}(\alpha - \beta)^2}
{\sqrt{2\alpha^2 + 2\beta^2} + {\rm sign} (\theta - \tau)(\alpha + \beta)}}.\nonumber
\end{align}
\end{subequations}
It is known from Theorem~\ref{alphabeta} that
$sign(\alpha)=sign(\beta)=sign(\theta-\tau)$ for $\epsilon$ reasonably small.
Consequently,
we obtain $sign(\theta-\tau)(\alpha+\beta)>0$. Therefore,
\begin{equation}\label{equation:14}
\|y-x\|\leq\frac{1}{|\sigma-\tau|} |\alpha-\beta|.
\end{equation}
From \eqref{alphabetabound}, \eqref{equation:13} and (\ref{equation:14}) we obtain
\begin{equation}\label{equation:46}
\frac{\|y-x\|}{\|y\|}\leq\frac{|\gamma||\alpha-\beta|}{|\sigma-\tau|\sin_{\max}
{\rm sep}(\gamma,L)}\leq \frac{2|\gamma|\|A\|(\|s\|+\|t\|)}{|\sigma-\tau|{\rm sep}
(\gamma,L)}=\mathcal{O}({\epsilon}).
\end{equation}

Summarizing the above derivation, we have established
the following results.

\begin{theorem}\label{theorem:4}
  Assume that $\sin\varphi=\mathcal{O}({\epsilon})$,
  $\sin\psi=\mathcal{O}({\epsilon})$,  $\|s\|=\mathcal{O}({\epsilon})$
  and $\|t\|=\mathcal{O}({\epsilon})$. If $\epsilon$ is sufficiently small
  such that  $\|s\|+\|t\|< \frac{|\sigma-\tau|{\rm sep}(\gamma,L)}{2|\gamma|\|A\|}$,
  we have
  \begin{equation}\label{equation:16}
\frac{1}{1+\frac{2|\gamma|\|A\|(\|s\|+\|t\|)}{|\sigma-\tau|{\rm sep}(\gamma,L)}}
\leq\delta
\leq\frac{1}{1-\frac{2|\gamma|\|A\| (\|s\|+\|t\|)}{|\sigma-\tau|{\rm sep}(\gamma,L)}},
\end{equation}
i.e.,
\begin{equation}\label{equation:52}
1-\mathcal{O}({\epsilon}) \leq \delta\leq  1+\mathcal{O}({\epsilon}).
\end{equation}
\end{theorem}

\begin{proof}
Combining \eqref{lowup} and \eqref{equation:46} gives
\begin{displaymath}
\left|1-\frac{1}{\delta}\right|\leq\frac{2|\gamma|\|A\|(\|s\|+\|t\|)}
{|\sigma-\tau|{\rm sep}(\gamma,L)}
=\mathcal{O}({\epsilon}),
\end{displaymath}
which means \eqref{equation:16} and \eqref{equation:52}.
\end{proof}

This theorem indicates that, in practical implementations, we can take $\delta=1$
when determining $\varepsilon$ from a given $\widetilde{\varepsilon}$.

\section{JDSVD type methods with deflation}\label{section:7}

Suppose that the $\ell$ singular values $\sigma_i$ closest to the target $\tau$
and the associated singular vectors $u_i$ and $v_i$ are required, $i=1,2,\dots,\ell$. We
assume that the $\sigma_i$ satisfy
\begin{equation}\label{equation:mult}
|\sigma_1-\tau|< |\sigma_2-\tau| < \dots < |\sigma_{\ell}-\tau|
<|\sigma_{\ell+1}-\tau|\leq\cdots\leq|\sigma_N-\tau|.
\end{equation}
Our JDSVD type methods
described in Section \ref{section:2} combined with some
deflation technique can meet the demands,
and the theoretical results established in Section \ref{section:3}
can be extended to the variants.

In this section, we assume that $k<\ell$ approximate singular values
and the corresponding approximate singular vectors written as
\begin{equation}
(\Theta_c,U_c,V_c)=\left(
\begin{bmatrix}\begin{smallmatrix}\theta_{(1,c)}&&\\&\ddots&\\&&\theta_{(k,c)}
\end{smallmatrix}\end{bmatrix},
[u_{(1,c)},\dots,u_{(k,c)}],[v_{(1,c)},\dots,v_{(k,c)}]\right)
\end{equation}
have already converged to the desired singular triplets $(\sigma_i,u_i,v_i),\
i=1,2,\ldots,k$ with the $\sigma_i$ labeled as \eqref{equation:mult}, that is,
the norms of the approximate singular triplets satisfy
\begin{equation}\label{stoppingcretiria}
\|r_{(i,c)}\|=\left\|\begin{bmatrix}Av_{(i,c)}-\theta_{(i,c)}
u_{(i,c)}\\A^Tu_{(i,c)}-\theta_{(i,c)} v_{(i,c)}\end{bmatrix}\right\|
\leq \|A\|_1\cdot tol, \quad\quad i=1,\dots,k,
\end{equation}
where $tol$ is the user prescribed tolerance.
The columns of $U_c$ and $V_c$ are orthogonal and $(\Theta_c,U_c,V_c)$
is the approximation to the partial SVD of $A$:
\begin{equation}\label{partialsvd}
(\Sigma_k,U_k,V_k)=\left(
\begin{bmatrix}\begin{smallmatrix}\sigma_{1}&&\\&\ddots&\\&&\sigma_{k}
\end{smallmatrix}\end{bmatrix}
[u_{1},\dots,u_{k}],[v_{1},\dots,v_{k}]\right).\end{equation}

Our goal is to compute the next singular triplet
$(\sigma_{k+1},u_{k+1},v_{k+1})$, which, is denoted as
$(\sigma_,u^{*},v^{*})$ for simplicity, as done in Section \ref{introduction}.

\subsection{Deflation}

The deflation applied to our JDSVD type methods works as follows.
Given a pair of $m$-dimensional searching subspaces $\mathcal{U}$ and $\mathcal{V}$
orthogonal to $U_c$ and $V_c$,
we use each of the extraction methods described in
Sections~\ref{subsection:HJDSVD}--\ref{subsection:RHJDSVD}
to extract an approximate singular triplet
$(\theta, u, v)$ satisfying the double orthogonality condition,
whose residual $r=r(\theta, u, v)$ is defined by \eqref{residual}.
When expanding the subspaces, different from \eqref{equation:1},
we instead solve a new correction equation of form
\begin{equation}\label{correction}
  \begin{bmatrix}\begin{matrix}I_M-  Q  Q^T& \\ &
 I_N-  Z  Z^T\end{matrix}\end{bmatrix}
 \begin{bmatrix}\begin{matrix}-\tau I_M & A\\A^T& -
 \tau I_N\end{matrix}\end{bmatrix}
 \begin{bmatrix}\begin{matrix}I_M-  Q  Q^T& \\ &
 I_N-  Z  Z^T\end{matrix}\end{bmatrix}
 \begin{bmatrix}\begin{matrix}s\\t\end{matrix}\end{bmatrix}=-r,
 \end{equation}
for $(s,t)\perp\perp(Q,Z)$, where the columns of $Q=[U_c,u]$ and
$Z=[V_c,v]$ are orthonormal. Then we orthonormalize $s$ and $t$
against $U$ and $V$, whose columns are orthonormal bases
of $\mathcal{U}$ and $\mathcal{V}$, respectively,
to obtain the expansion vectors
$u_{+}$ and $v_{+}$.
Expand $U$ and $V$ to $U_{+}=[U,u_{+}]$ and $V_{+}=[V,v_{+}]$,
respectively, whose columns form orthonormal bases of the expanded
searching subspaces $\mathcal{U}_{+}$ and $\mathcal{V}_{+}$,
from which we continue to extract a new approximate singular
triplet.

Obviously, new approximate left and right singular vectors
are always orthogonal to the already converged $U_c$ and $V_c$
since the latter ones are orthogonal to $\mathcal{U}_{+}$ and $\mathcal{V}_{+}$,
respectively. Such orthogonality is guaranteed to working
precision in finite precision arithmetic, provided that
$s$ and $t$ are orthonormalized against $U$ and $V$ to working precision
when expanding $\mathcal{U}$ and $\mathcal{V}$ to
$\mathcal{U}_{+}$ and $\mathcal{V}_{+}$, respectively.

However, we notice that,
under our assumption on $\mathcal{U}$ and $\mathcal{V}$,
although the current approximate singular vectors $(\theta,u,v)$ naturally
satisfies $(u,v)\perp\perp(U_c,V_c)$, the relation
$r(\theta,u,v)\perp\perp (U_c,V_c)$ holds only when $(U_c,V_c)=(U_k,V_k)$.
In order to guarantee the consistency of
the correction equation \eqref{correction}, we should use the projected residual
\begin{equation}\label{replace}
r_p=\begin{bmatrix}\begin{smallmatrix}I_M-U_cU_c^T&\\
 &I_N-V_cV_c^T\end{smallmatrix}\end{bmatrix}r(\theta,u,v)
\end{equation}
to replace the right-hand side $r(\theta,u,v)$ in \eqref{correction}.

%Notice that current $m$ dimensional searching
%subspaces $\mathcal{U}$ and $\mathcal{V}$ contain reasonable information
%on the next desired $(u^*,v^*)$. Therefore, when computing
%$(\sigma,u^*,v^*)$, we may benefit a lot from $\mathcal{U}$ and $\mathcal{V}$
%and find a better initial approximation to $(u^*,v^*)$ rather than the one
%generated in some random way.

Notice that current $m$ dimensional searching
subspaces $\mathcal{U}$ and $\mathcal{V}$ contain reasonable information
on the next desired $(u^*,v^*)=(u_{k+1},v_{k+1})$. Therefore, when computing
$(\sigma_{k+1},u_{k+1},v_{k+1})$ with the current $(\theta,u,v)$
already converged to $(\sigma_k,u_k,v_k)$,
we may benefit a lot from $\mathcal{U}$ and $\mathcal{V}$
and find a better initial approximation to $(u^*,v^*)$ rather than the one
generated in some random way.
We will purge the newly converged left and
right singular vectors from
$\mathcal{U}$ and $\mathcal{V}$, and obtain the new $(m-1)$
dimensional searching subspaces which are orthogonal to
the converged left and right singular vectors, respectively, from which
we use harmonic or refined harmonic extraction to compute
approximate left and right singular vectors as the initial approximation
to the desired $(u^*,v^*)$. Then we proceed
to expand the searching subspaces in the regular way as
done in Algorithm~\ref{algorithm:1}.

Precisely, we formally construct the desired $(m-1)$ dimensional left
and right searching subspaces as follows:
Let the columns of $U$ and $V$ form the orthonormal bases of $\mathcal{U}$
and $\mathcal{V}$, and the converged left and right singular
vectors $u=Uc$ and $v=Vd$ with $\|c\|=\|d\|=1$.
Then we augment $c$ and $d$ such that $[c,C]$
and $[d,D]$ are orthogonal, in which the $m\times (m-1)$
orthonormal $C$ and $D$ are obtained by computing
the full QR factorizations of the $m\times 1$ matrices $c$ and $d$
at cost of $\mathcal{O}(m^3)$ flops. The orthonormal columns of
$U_{new}=UC$ and $V_{new}=VD$ form bases of the desired
$(m-1)$ dimensional searching subspaces $\mathcal{U}_{new}$
and $\mathcal{V}_{new}$, which are orthogonal to the converged
$u$ and $v$, respectively. Computationally, however, we do
not need to form $U_{new}=UC$ and $V_{new}=VD$ and
project $A$ onto $\mathcal{U}_{new}$ and $\mathcal{V}_{new}$ explicitly,
which is somehow expensive.
%The key is that what we need is
%to form the new projection matrix
%$H:=U_{new}^TAV_{new}=C^T(U^TAV)D=C^T H D$. Note that $H$ is already
%available when computing the converged approximation
%to the $k$th left and right singular vector pair $(u_k,v_k)$ of $A$.
%Therefore, we only need to compute $H:=C^T H D$,
%whose cost is negligible,
%compared with the explicit computation of $U_{new}^TAV_{new}$.
The key is that what we need is
to form the new matrices
\begin{eqnarray*}
H_{new}&:=&U_{new}^TAV_{new}=C^T(U^TAV)D=C^T H D,\\
G^{(1)}_{new}&:=&U_{new}^TAA^TU_{new}=C^T(U^TAA^TU)C=C^TG^{(1)} C,\\
G^{(2)}_{new}&:=&V_{new}^TA^TAV_{new}=D^T(V^TA^TAV)D=D^TG^{(2)} D.
\end{eqnarray*}
Note that $H$, $G^{(1)}$ and $G^{(2)}$ are already
available when computing the converged approximation
to the $(k+1)$th left and right singular vector pair $(u_{k+1},v_{k+1})$ of $A$.
Therefore, we only need to update $H=C^T H D$, $G^{(1)}=C^TG^{(1)} C$ and
$G^{(2)}=D^TG^{(2)} D$,
whose costs are negligible,
compared with the explicit computations of $U_{new}^TAV_{new}$, $U_{new}^TAA^TU_{new}$
and $V_{new}^TA^TAV_{new}$.

\subsection{Theoretic extensions}

Assuming that the first $k$ singular triplets $(\sigma_i,u_i,v_i)$ have
been computed exactly, i.e.,
\begin{equation}\label{assumption}
(\Theta_c,U_c,V_c)=(\Sigma_k,U_k,V_k),
\end{equation}
we will prove that the theory established in Section~\ref{section:3} works
for the JDSVD methods with deflation described above.
That is, as far as the correction equation \eqref{correction} with the
right-hand side \eqref{replace} is concerned, there are the same
relationships between
the solution accuracy $\varepsilon$ defined by \eqref{equation:2}
and the accuracy $\widetilde \varepsilon$ of the expansion vectors
$\widetilde{s}$ and $\widetilde{t}$ defined by \eqref{widevarepsion}.

First of all, we need to prove that
the solution $[s^T,t^T]^T$ of the correction equation \eqref{correction}
still has the expression \eqref{equation:3}. By using the same derivation
as \eqref{equation:20} and noticing that both $(s,t)$ and $(u,v)$
are double orthogonal to $(U_c,V_c)$, it is direct to
justify that this is true.

With Lemma \ref{lemma2} and the expression \eqref{equation:3}, we can
extend Theorem~\ref{theorem:2} to the solution accuracy $\varepsilon$
of \eqref{correction} and the accuracy $\widetilde \varepsilon$ of
the expansion vectors $\widetilde{s}$ and $\widetilde{t}$.
To this end, we only need to modify the
proof followed slightly:
When decomposing $u$ and $v$ into the orthogonal direct sums
\eqref{equation:10}, we have $(p,q)\perp\perp(U_c,V_c)$ since
both $(u,v)$ and $(u_{*},v_{*})$ are double orthogonal to $(U_c,V_c)$.
As a result, for the second term in the right-hand side
of \eqref{equation:22} we have
\begin{equation}
  \left\|B\begin{bmatrix}\alpha \sin\varphi p\\
 \beta\sin\psi q\end{bmatrix}\right\| \leq\frac{1}{|\sigma-\tau|}
 \sqrt{\alpha^2\sin^2\varphi + \beta^2\sin^2\psi},
 \footnote{Since span$\{U_k,V_k\}$ is an invariant subspace of $B$,
 it is easy to verify that
 $\|Bx\|\leq\frac{1}{|\sigma-\tau|}\|x\|$ for any $x\perp\perp(U_k,V_k)$.}
\end{equation}
starting with which we repeat the remaining proof of Theorem \ref{theorem:2}
step by step and extend the theorem to the solution of \eqref{correction}.

To make Theorem \ref{theorem:2} clearer, apart from Theorem \ref{alphabeta},
which trivially holds for $\alpha$, $\beta$ defined by \eqref{equation:3} and
$\gamma$ defined by \eqref{gamma},
we need the following lemma, which is
a generalization of Lemma \ref{lemma:1} and reduces to
it when $k=0$.

 \begin{lemma}\label{lemma:2}
Let $\left(\frac{1}{\sigma-\tau},w=\frac{1}{\sqrt{2}}
\begin{bmatrix}\begin{smallmatrix}u^*\\v^*
 \end{smallmatrix}\end{bmatrix}\right)$
 be a simple eigenpair of $B$ and $(\Sigma_k,U_k,V_k)$ be defined as
 \eqref{partialsvd}, and let
 $W=\begin{bmatrix}\begin{smallmatrix}U_k&\\&V_k  \end{smallmatrix}\end{bmatrix}$
 and $[w,W,W_{\perp}]$ be orthogonal. Then
 \begin{equation}\label{eqlemme5}
 \begin{bmatrix} w^T\\W^T\\ W_{\perp}^T \end{bmatrix}
 B[w, W, W_{\perp}]
 =\begin{bmatrix} \frac{1}{\sigma-\tau}&&\\&
 \begin{bmatrix}\begin{smallmatrix}-\tau I_k &\Sigma_k\\ \Sigma_k& -
 \tau I_k\end{smallmatrix}
 \end{bmatrix}^{-1} & \\&&L \end{bmatrix},
 \end{equation}
 where $L=W_{\perp}^TBW_{\perp}$.
 Suppose that $\left(\gamma,z=\frac{1}{\sqrt{2}}
 \begin{bmatrix}\begin{smallmatrix}u\\v \end{smallmatrix}\end{bmatrix}\right)$
 is an approximation to $(\frac{1}{\sigma-\tau},w)$ satisfying $z\perp\perp (U_k,V_k)$
and $\gamma$ is not an eigenvalue of $L$. Then
 \begin{equation}\label{angle1}
 \sin\angle(z,w)\leq\frac{\|Bz-\gamma z\|}{{\rm sep}(\gamma,L)}.
 \end{equation}
 \end{lemma}
 \begin{proof}
 Since $z\perp\perp (U_k,V_k)$, we have $W^Tz=0$
 and
 \begin{equation}\label{decomz}
 z=ww^Tz+WW^Tz+W_{\perp}W_{\perp}^Tz=ww^Tz+W_{\perp}W_{\perp}^Tz.
 \end{equation}
 By \eqref{eqlemme5}, since $W_{\perp}^Tw=0$ and the smallest singular value
 of $L-\gamma I$ is ${\rm sep}(\gamma,L)$,
 it holds that
 \begin{eqnarray}
   \|(B-\gamma I)z\|&=&\|(\frac{1}{\sigma-\tau}-\gamma)
   (w^Tz)w+W_{\perp}(L-\gamma I)W_{\perp}^Tz\| \nonumber\\
   &\geq& \|W_{\perp}(L-\gamma I)W_{\perp}^Tz\|= \|(L-\gamma I)W_{\perp}^Tz\|
   \geq {\rm sep}(\gamma,L) \|W_{\perp}^Tz\|. \label{inequ}
 \end{eqnarray}
It follows from \eqref{decomz} that $\|W_{\perp}^Tz\|=
\|W_{\perp}W_{\perp}^Tz\|=\|(I-ww^T)z\|=\sin\angle(z,w)$,
which and \eqref{inequ} establish \eqref{angle1}.
 \end{proof}

With the help of Lemma \ref{lemma:2}, it is straightforward to derive
Theorem \ref{theorem:3} for the solution accuracy $\varepsilon$ of
\eqref{correction} and the
accuracy $\widetilde \varepsilon$ of the expansion vectors,
where $\delta$ satisfies the estimates in Theorem \ref{theorem:4}.

It is necessary to point out that our theoretical extensions above are
established under the assumption \eqref{assumption}.
For $tol>0$, we can prove that the bounds in
Theorem~\ref{theorem:2} and Theorems~\ref{theorem:3}--\ref{theorem:4}
hold within $\mathcal{O}(tol)$. The derivation
is routine but tedious, and we do not give details in this paper.

\section{Practical stopping criteria for inner iterations}\label{section:4}
In this section, we use the results established in the last two sections
to determine a practical $\varepsilon$ from a given $\widetilde{\varepsilon}$
and derive practical stopping criteria for the inner iterations involved in the
inexact JDSVD type algorithms.
To this end, we will restore the notations used for HJDSVD and RHJDSVD
in Section \ref{section:2}
in order to treat the two inexact algorithms separately.

From Theorem \ref{theorem:3}, we take $\delta=1$ in computation.
Since we cannot compute $\sqrt{\|g_{\perp}\|^2+\|h_{\perp}\|^2}$ directly,
we simply replace it by its upper bound one, which makes $\varepsilon$
as small as possible, so that the inexact JDSVD methods are more reliable to mimic
their exact counterparts.
Also, in Theorem~\ref{theorem:3}, we replace the a-prior quantity
$\sigma$ by the current approximate singular value,
i.e., $\rho$ for HJDSVD and $\rho^{\prime}$ for RHJDSVD.

For ${\rm sep}(\gamma,L)$, we see from Theorem~\ref{alphabeta}
that $\gamma\approx \frac{1}{\alpha}
\approx\frac{1}{\beta}= \frac{1}{\theta-\tau}+ \mathcal{O}(\epsilon^2)$.
So we use ${\rm sep}(\frac{1}
{\rho-\tau},L)$ and ${\rm sep}(\frac{1}{\rho^{\prime}-\tau},L)$ to estimate
${\rm sep}(\gamma,L)$ in HJDSVD and RHJDSVD, respectively. However, since $L$
is unavailable, it is impossible to compute
${\rm sep}(\frac{1}{\rho-\tau},L)$ or ${\rm sep}(\frac{1}{\rho^{\prime}-\tau},L)$.
 We can exploit the eigenvalues of the matrix pencil
$(F,G)$ to estimate ${\rm sep}(\frac{1}{\rho-\tau},L)$ or
${\rm sep}(\frac{1}{\rho^{\prime}-\tau},L)$. Let $\frac{1}{\nu_i}$ be
the eigenvalues of $(F,G)$ in (\ref{equation:17}) other than
its largest $\frac{1}{\nu}$ in magnitude
and such that $\theta^{\prime}_i=\nu_i+\tau$ are positive.
Then $\theta^{\prime}_i=\nu_i+\tau$ are approximate singular values of $A$.
Therefore, we have the estimates
\begin{displaymath}\begin{aligned}
{\rm sep}(\frac{1}{\rho-\tau},L) &\approx \min_{\theta^{\prime}_i\not=\rho}
|\frac{1}{\rho-\tau} -\frac{1}{\theta^{\prime}_i-\tau}|  \quad  \quad \ \mbox{for\ HJDSVD},\\
{\rm sep}(\frac{1}{\rho^{\prime}-\tau},L) &\approx\min_{\theta^{\prime}_i\not=\rho}|
\frac{1}{\rho^{\prime}-\tau}-\frac{1}{\theta^{\prime}_i-\tau}| \quad  \quad \mbox{for\ RHJDSVD}.
\end{aligned}\end{displaymath}

Taking the equality in the bound of (\ref{equation:15}) with all the above estimates,
we obtain
\begin{equation}
\varepsilon_{H}
=2\sqrt2\widetilde\varepsilon\max_{\nu_i\not=\nu}
\frac{|\nu_i|} {|\nu_i+\tau-\rho|}
\end{equation}
for HJDSVD and
\begin{equation}
\varepsilon_{RH}
=2\sqrt2\widetilde\varepsilon\max_{\nu_i\not=\nu}
\frac{|\nu_i|}{|\nu_i+\tau-\rho^{\prime}|}
\end{equation}
for RHJDSVD, respectively.

Based on the above results, we can now write $\varepsilon\leq c\widetilde\varepsilon$
in a unified form, where
\begin{equation}
c=\left\{\begin{aligned}
&2\sqrt2\max_{\nu_i\not=\nu}
\frac{|\nu_i|}{|\nu_i+\tau-\rho|} \quad  \quad \ \mbox{for HJDSVD,} \\
&2\sqrt2\max_{\nu_i\not=\nu}
\frac{|\nu_i|}{|\nu_i+\tau-\rho^{\prime}|} \quad  \quad \mbox{for RHJDSVD }\end{aligned}\right.
\end{equation}
for $m>1$, and $c=1$ for $m=1$ when solving (\ref{equation:1}) approximately.

For a not very small $\widetilde\varepsilon$, we may have $\varepsilon\geq1$
in case $c$ is large,
which will make $(\widetilde s,\widetilde t)$ have no accuracy as an approximation
to the exact solution $(s,t)$ to (\ref{equation:1}), so that
$\widetilde {\mathcal{U}}_{+}$ and $\widetilde{\mathcal{V}}_{+}$ may have no
improvement over $\mathcal{U}$ and $\mathcal{V}$. As a guard remedy,
in order to make
$\widetilde {\mathcal{U}}_{+}$ and $\widetilde{\mathcal{V}}_{+}$ have some
improvements, we propose to use
\begin{equation}\label{equation:39}
\varepsilon\leq \min\{c\widetilde \varepsilon,0.01\}.
\end{equation}

However,
$\varepsilon=\frac{\left\| \begin{bmatrix} \begin{smallmatrix} \widetilde{s}\\
\widetilde {t} \end{smallmatrix} \end{bmatrix}- \begin{bmatrix}
\begin{smallmatrix}s\\t \end{smallmatrix} \end{bmatrix}\right\|}
{\left\| \begin{bmatrix} \begin{smallmatrix}s\\t \end{smallmatrix}
\end{bmatrix}\right\|}$
is an a-prior error and uncomputable in practice, so that we cannot
determine whether or not \eqref{equation:39} is fulfilled
for a prescribed $\widetilde\varepsilon$.
Nevertheless, it is easy to justify that
 \begin{equation}\label{equation:40}
 \frac{\varepsilon}{\kappa(B^{\prime})}
 \leq r_{in}
 \leq \kappa(B^{\prime})\varepsilon,
 \end{equation}
 where
 \begin{equation}\label{equation:47}
 r_{in}=\frac{1}{\|r\|}\left\|-r\!-\!
 \begin{bmatrix}
  I_M\!-\!P_{u}&\\&\!\!\!\!\!I_N\!-\!P_{v}
 \end{bmatrix}\begin{bmatrix} \!-\tau I_M&\!A\\\!A^T&\!-\tau I_N
   \end{bmatrix}\! \begin{bmatrix}
 I_M\!-\!P_{u}&\\ &\!\!\!\!\!I_N\!-\!P_{v}   \end{bmatrix}\!
 \begin{bmatrix}  \widetilde s\\\widetilde t
  \end{bmatrix}\right\|
\end{equation}
 is the computable relative residual norm of
 the approximate solution $[\widetilde s^T,\widetilde t^T]^T$
 of (\ref{equation:1}),
 and $\kappa(B^{\prime})=
 \| B^{\prime}\|\|(B^{\prime})^{-1}\|$ with
 $$B^{\prime}=B|_{(Q,Z)^{\perp\perp}}=\begin{bmatrix} -
 \tau I_M & A\\A^T& -\tau I_N \end{bmatrix}^{-1}
 \Big|_{(Q,Z)^{\perp\perp}}$$
  being the restriction of $B$ to the double orthogonal complement\footnote{For
  arbitrary $Q\in \mathbb{R}^{M\times l},
  Z\in \mathbb{R}^{N\times l}$, the double orthogonal complement of $(Q,Z)$
  is defined as $(Q,Z)^{\perp\perp}:=\{(q,z)|q\in
  \mathbb{R}^M, z\in \mathbb{R}^N, q\perp Q,z\perp Z\}$.} of $(Q,Z)$.
  Based on the two bounds in (\ref{equation:40}),
  we practically stop the inner iterations at each outer iteration when
\begin{equation}\label{equation:41}
r_{in}\leq \min\{c\widetilde \varepsilon,0.01\}
\end{equation}
for a given $\widetilde \varepsilon$. (\ref{equation:40}) indicates that $r_{in}$
is a reasonable replacement for $\varepsilon$
when $\kappa(B^{\prime})$ is fairly modest. We should remark that the lower and
upper bounds in \eqref{equation:40}
are the estimates in the worst case for the a-posterior relative residual
norm $r_{in}$ in terms of the a-prior relative error $\varepsilon$.

Let us have a closer look at $\kappa(B^{\prime})$, whose size has two effects:
(i) it decides the convergence speed of the Krylov iterative solver MINRES \cite{saad03},
and the larger it is, the more  slowly MINRES converges
generally \cite{greenbaum}; (ii) it decides how
the a-posterior relative residual norm $r_{in}$ differs from the a-prior error
$\varepsilon$.
The smaller $\kappa(B^{\prime})$ is, the more reliable \eqref{equation:41} is;
conversely, for $\kappa(B^{\prime})$ large, the a-posterior replacement
\eqref{equation:41} may not be reliable. Notice that,
when $(u,v)=(u^{*},v^{*})$ and $(U_c,V_c)=(U_k,V_k)$, we have
  \begin{equation}\label{equation:44}
  \kappa(B^{\prime})=\frac{ \sigma_{\max}(B^{\prime})}
  {\sigma_{\min}(B^{\prime})}=\frac{\max_{i=k+2,k+3,\dots,n}
  |\pm\sigma_i-\tau|}{\min_{i=k+2,k+3,\dots,n} |\pm\sigma_i-\tau|}
  =\frac{\sigma_{\max}+\tau}{|\sigma_{k+2}-\tau|}
\end{equation}
where $\sigma_{\max}=\max\{\sigma_{k+2},\sigma_{k+3},\ldots,\sigma_n\}$.
Therefore, the continuity tells us that
$\kappa(B^{\prime})\approx \frac{\sigma_{\max}+\tau}{|\sigma_{k+2}-\tau|}$
for $\sin\angle(u^{*},u)=\mathcal{O}(\epsilon)$ and
$\sin\angle(v^{*},v)=\mathcal{O}(\epsilon)$. This result shows that
the correction equation becomes (asymptotically) better conditioned as
$k$ increases.

\section{Numerical experiments}\label{section:5}

We report numerical experiments to confirm our theory.
Table \ref{table0} lists the test matrices from \cite{davis2011university}
together with some of their basic properties, where $\kappa(B^{\prime})$ is
the right-hand side of \eqref{equation:44} with $k=0$. For the matrices with $M<N$,
we apply the algorithms to their transposes. We aim to show
two points:
(\romannumeral1) for fairly small $\widetilde\varepsilon=10^{-3}$ and $10^{-4}$,
the non-restarted and restarted inexact HJDSVD and RHJDSVD algorithms behave
(very) like their exact counterparts;
(\romannumeral2) regarding the total inner iterations and overall efficiency,
the inexact JDSVD type algorithms are substantially more efficient than their
exact counterparts. We will compute the $\ell$ singular triplets for given $\tau$'s,
where we take $\ell=1$ and $\ell=5$; we will report the experiments
on two $\ell$'s, separately.

\begin{table}[tbhp]
{\footnotesize
\caption{Properties of test matrices, where $nnz(A)$ is the number of nonzero entries
in $A$, $\kappa(B^{\prime})$ is defined by \eqref{equation:44} for the chosen target
$\tau$ when $k=0$, and $\sigma_{\max}$ and $\sigma_2$ are
estimated by the MATLAB function {\sf svds.m}.
The notation $+\infty$ indicates that $A$ is rank deficient.}\label{table0}
\begin{center}
\begin{tabular}{|c|c|c|c|c|c|} \hline
{\rm Matrix}      & $M\times N$   &$nnz(A)$& $\|A\|$    &$\kappa(A)$
&$\kappa(B^{\prime})$         \\ \hline
{\rm deter4}      &$3235\times 9133$  &$19,231$    &$10.2$
&$3.71e\mbox{+}{2}$ &$13.7$      \\
{\rm lp\_bnl2}    &$2324\times4486$  &$14,996$    &$2.12e\mbox{+}2$
&$7.77e\mbox{+}{3}$ &$4.90e\mbox{+}{2}$       \\
{\rm r05 }        &$5190\times 9690$  &$104,145$    &$18.2 $          &$1.22e\mbox{+}{2}$
&$17.3$                  \\
{\rm large}       &$4282\times 8617$  &$20,635$    &$4.04e\mbox{+}3$ &$4.94e\mbox{+}5$
&$2.25e\mbox{+}3$         \\
{\rm gemat1}      &$4929\times 10595$ &$46,591$    &$2.34e\mbox{+}4$ &$1.17e\mbox{+}8$
&$9.97e\mbox{+}5$         \\
{\rm tomographic1}&$73159\times 59498$ &$647,495$   &$6.98$  &$+\infty$&$12.30$
       \\
{\rm watson\_1}   &$201155 \times 386992$&$1,055,093$&$20.59$ &$8.64e\mbox{+}2$    &$7.58e\mbox{+}3$
       \\
{\rm degme}       &$185501\times 659415$&$8,127,528$  &$2.24e\mbox{+}3$ & $5.42e\mbox{+2}$
&$1.52e+3$ \\ \hline
\end{tabular}
\end{center}
}
\end{table}

All the numerical experiments were performed on an Intel (R) Core (TM) i7-7700
CPU 3.60GHz with the main memory 8GB using the Matlab R2017a with the machine
precision $\epsilon_{\rm mach}=2.22\times 10^{-16}$ under the Windows 10 operating
system.

We denote by HJDSVD$(\widetilde\varepsilon)$ and RHJDSVD$(\widetilde\varepsilon)$ the
inexact JDSVD algorithms for a given $\widetilde\varepsilon$.
We use MINRES to solve the correction equation (\ref{equation:1}) or \eqref{correction}
by taking the ($M+N$)-dimensional zero vector
as an initial approximate solution. The code {\sf minres.m} is from Matlab R2017a.
At each outer iteration of HJDSVD or RHJDSVD, we stop
inner iterations when the stopping criterion \eqref{equation:41} is fulfilled.

For outer iterations, we always take the initial vectors $u_0$ and $v_0$ to be the
normalized $M$- and $N$-dimensional vectors whose elements are equal.
We restart outer iterations after the maximum dimensions of $\mathcal{U}$ and
$\mathcal{V}$ reach 20. The restarting technique used here is the thick-restart,
i.e., when restarting, instead of one, we compute the best three approximate triplets
$(\theta_i, u_i, v_i), i=1,2,3$, in steps 5-7 of
Algorithm \ref{algorithm:1}. To do that, we compute three eigenvectors
$f_i,i=1,2,3,$ of $(F, G)$ associated with the largest three
eigenvalues in magnitude, and each $f_i$ corresponds to an approximate
singular triplet $(\rho_i, \widetilde u_i,\widetilde v_i)$ for HJDSVD.
As for RHJDSVD, we compute the eigenvector $\widehat f_i$
associated with the smallest eigenvalue of $G^{\prime}$
by taking $\rho=\rho_i, \ i=1,2,3$
in \eqref{gprime}, and each $\widehat f_i$ corresponds to an
approximate singular triplet $(\rho^{\prime}_i, \widehat u_i,\widehat v_i)$
for RHJDSVD. Then in the next restart cycle
we use the computed three pairs of approximate left and right singular vectors
to construct new initial $\mathcal{U}$ and $\mathcal{V}$ of dimension three,
respectively, and expand them in the way described by Algorithm~\ref{algorithm:1}.

An approximate singular triplet $(\theta, u=Uc, v=Vd)$ obtained by
the JDSVD type algorithms is claimed to have
converged if the relative residual norm
\begin{displaymath}
\|r\|=\|r(\theta,u,v)\|\leq \|A\|_1\cdot tol=\|A\|_1\cdot 10^{-10}.
\end{displaymath}
We stop outer iterations
if all $\ell$ desired singular triplets have been found or the maximum
outer iterations have been used.

In each of the exact algorithms, for the experimental purpose, we have also solved
the correction equations by applying the LU factorization of
$\begin{bmatrix}\begin{smallmatrix}-\tau I_{M} & A\\A^T &-\tau I_{N}
\end{smallmatrix}\end{bmatrix}$ to \eqref{equation:3},
where $\alpha$ and $\beta$ are computed by the requirement of
double orthogonality of $(s,t)$ and $(u,v)$.
However, due to excessive storage and/or computational cost,
we must remind that it is generally
unrealistic to perform the LU factorization
when $A$ is really large, as confirmed by our experiments.
So in the exact JDSVD type algorithms
it is practical to solve the correction equation
(\ref{equation:1}) or \eqref{correction} by iterative solvers, in which
we stop inner iterations when
\begin{equation}\label{equation:42}
r_{in}\leq10^{-14},
\end{equation}
where $r_{in}$ is defined by (\ref{equation:47}), indicating as if
(\ref{equation:1}) or \eqref{correction} is solved exactly in finite precision
arithmetic. In this case, we call the JDSVD algorithms  iterative exact,
denoted by ''Iter. Exact''.

In all the tables, we denote by $I_{out}$ and $I_{in}$ the
total numbers of outer and inner iterations, respectively,
and by $T_{cpu}$ the CPU time (in seconds) counted by the
Matlab recommended commands {\sf tic} and {\sf toc}. We point out that
$I_{in}$ is a reasonable measure of the overall efficiency of
the JDSVD type algorithms, and it equals the total
number of the matrices $A$ and $A^T$-vector products used by MINRES.
In contrast, $T_{cpu}$ may be an unreliable measure of overall
efficiency since it heavily depends on many factors,
such as the computer used, the programming language used,
the programming optimization,
the computing environment, and the operating system.

\emph{We first test the matrix $A=$ deter4 with $\tau=7$. The desired
singular value $\sigma^*\approx5.74$ is an interior singular value of $A$.}

\begin{figure}[tbhp]
\centering
\subfloat[]{\label{fig1a}\includegraphics[width=0.48\textwidth]{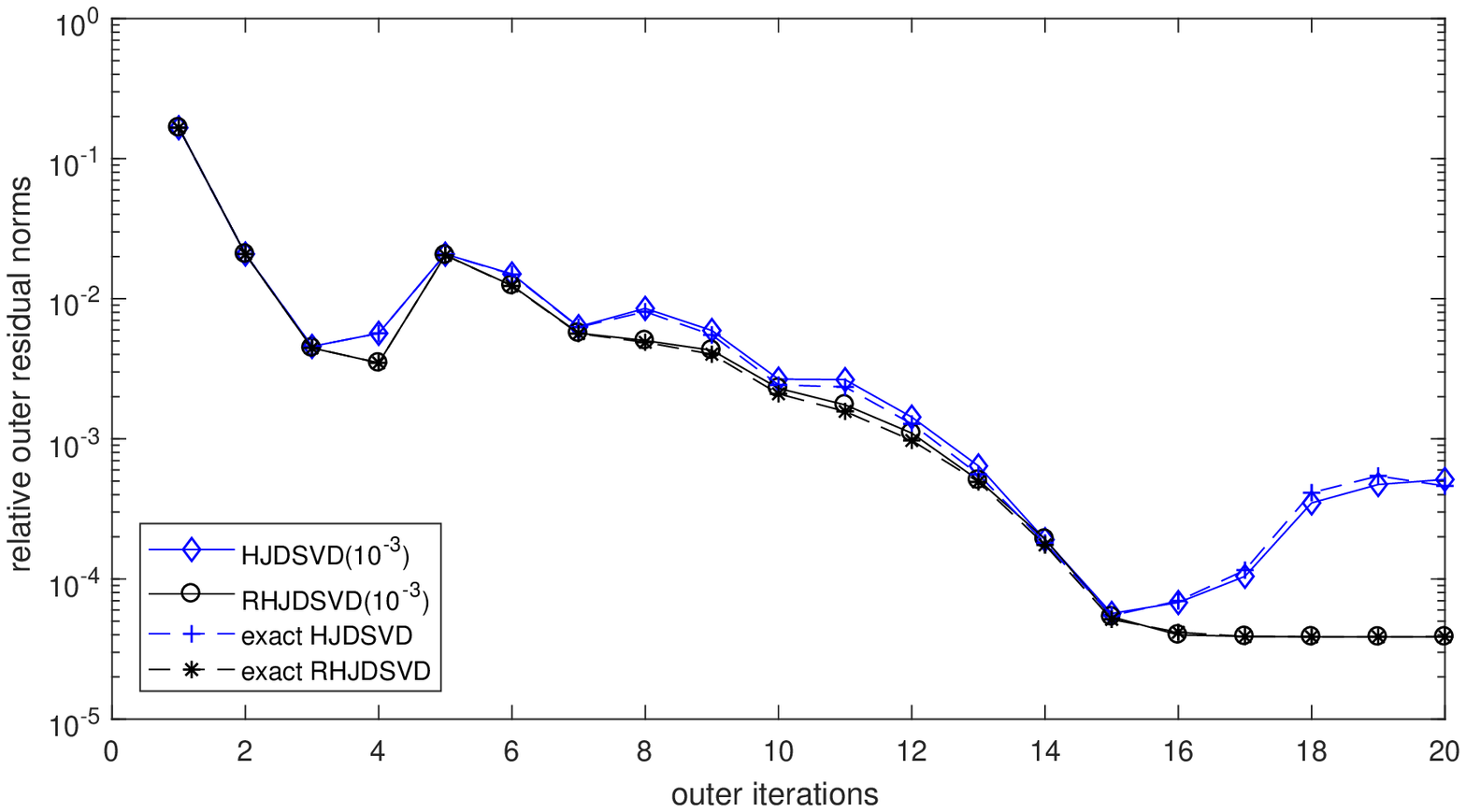}}
\ \ \ \
\subfloat[]{\label{fig1b}\includegraphics[width=0.48\textwidth]{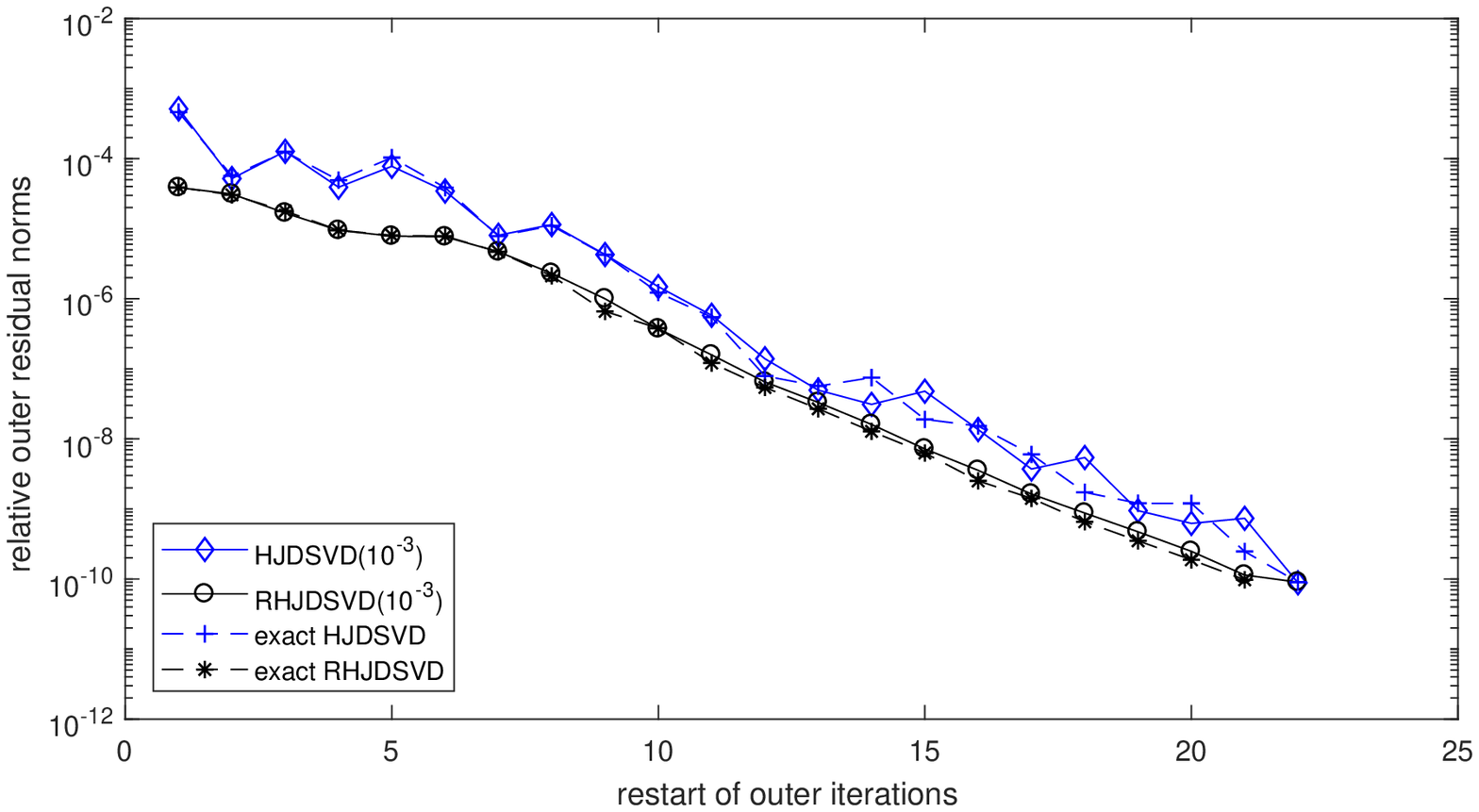}}
\caption{deter4 with $\tau=7$.}
\label{fig1}
\end{figure}
\begin{table}[tbhp]
{\footnotesize
\caption{deter4 with $\tau=7$.}\label{table1}
\begin{center}
\begin{tabular}{|c|c|c|c|c|c|c|c|} \hline
\multicolumn{4}{|c|}{Accuracy: $\widetilde\varepsilon=10^{-3}$}
&\multicolumn{4}{|c|} {Accuracy: Iter. Exact}         \\ \hline
Algorithm      &$I_{out}$    &$I_{in}$    &$T_{cpu}$    &Algorithm     &$I_{out}$   &$I_{in}$      &$T_{cpu}$  \\ \hline
\mbox{HJDSVD}  &430          &7800        &3.08         &\mbox{HJDSVD} &427         &36776         &10.85 \\
\mbox{RHJDSVD} &423          &5673        &2.47         &\mbox{RHJDSVD}&418         &34843         &10.56  \\ \hline
\multicolumn{4}{|c|}{ Accuracy: $\widetilde\varepsilon=10^{-4}$}
&\multicolumn{4}{|c|}{ Accuracy: LU factorization}    \\ \hline
Algorithm      &$I_{out}$    &$I_{in}$    &$T_{cpu}$    &Algorithm     &$I_{out}$  &$I_{in}$      &$T_{cpu}$\\ \hline
\mbox{HJDSVD}  &424          &7722        & 2.91        &\mbox{HJDSVD} &427        &$-$           &37.2 \\
\mbox{RHJDSVD} &421          &5728        &2.34         &\mbox{RHJDSVD}&418        &$-$           &36.6 \\ \hline
\end{tabular}
\end{center}
}
\end{table}

Table~\ref{table1} lists the results obtained. Clearly,
all the algorithms for $\widetilde\varepsilon=10^{-3},10^{-4}$
are successful in computing $\sigma$ and its corresponding singular vectors.
Figure~\ref{fig1}(a) and (b) depict the convergence curves of the inexact and exact
JDSVDs during the first cycle and all cycles, respectively.
We see that the restarted and non-restarted HJDSVD($10^{-3}$)
and RHJDSVD($10^{-3}$) behave quite like their exact counterparts, respectively.

From Table \ref{table1}, we observe that each inexact algorithm uses almost
the same outer iterations as its exact counterpart does. This indicates
that for $\widetilde\varepsilon=10^{-3},10^{-4}$ the inexact JDSVD type
algorithms mimic their exact versions very well. However, regarding
the overall efficiency, compared with their iterative exact versions,
we see that the inexact JDSVD algorithms cost only less than
$22\%$ of total inner iterations, or less than $28\%$ of
CPU time, to compute the desired singular triplet.
A smaller $\widetilde\varepsilon$ is unnecessary
since it cannot reduce outer iterations and in the meantime
we have to solve the correction equations with higher accuracy, which
will increase the total cost substantially.

We can also see from Figure~\ref{fig1} and Table~\ref{table1}
that the restarted RHJDSVD algorithm uses fewer outer iterations than
the restarted HJDSVD algorithm and improves the overall efficiency,
as seen from $I_{in}$ and $T_{cpu}$. Actually, we have observed from
Figure~\ref{fig1} (a) that (exact and inexact) RHJDSVD computed more accurate
approximate singular triplets than HJDSVD at the 16th to the 20th outer iterations,
and the latter diverged and delivered less accurate approximations,
which confirms the better convergence of the refined harmonic extraction and
the possible irregular convergence of the harmonic extraction.

\emph{We next test the matrix $A=$ lp\_bnl2 with $\tau=8.16$. The desired singular
value $\sigma^*\approx7.71$ is an interior one of $A$.}

\begin{figure}[tbhp]
\centering
\subfloat[]{\label{fig2a}\includegraphics[width=0.48\textwidth]{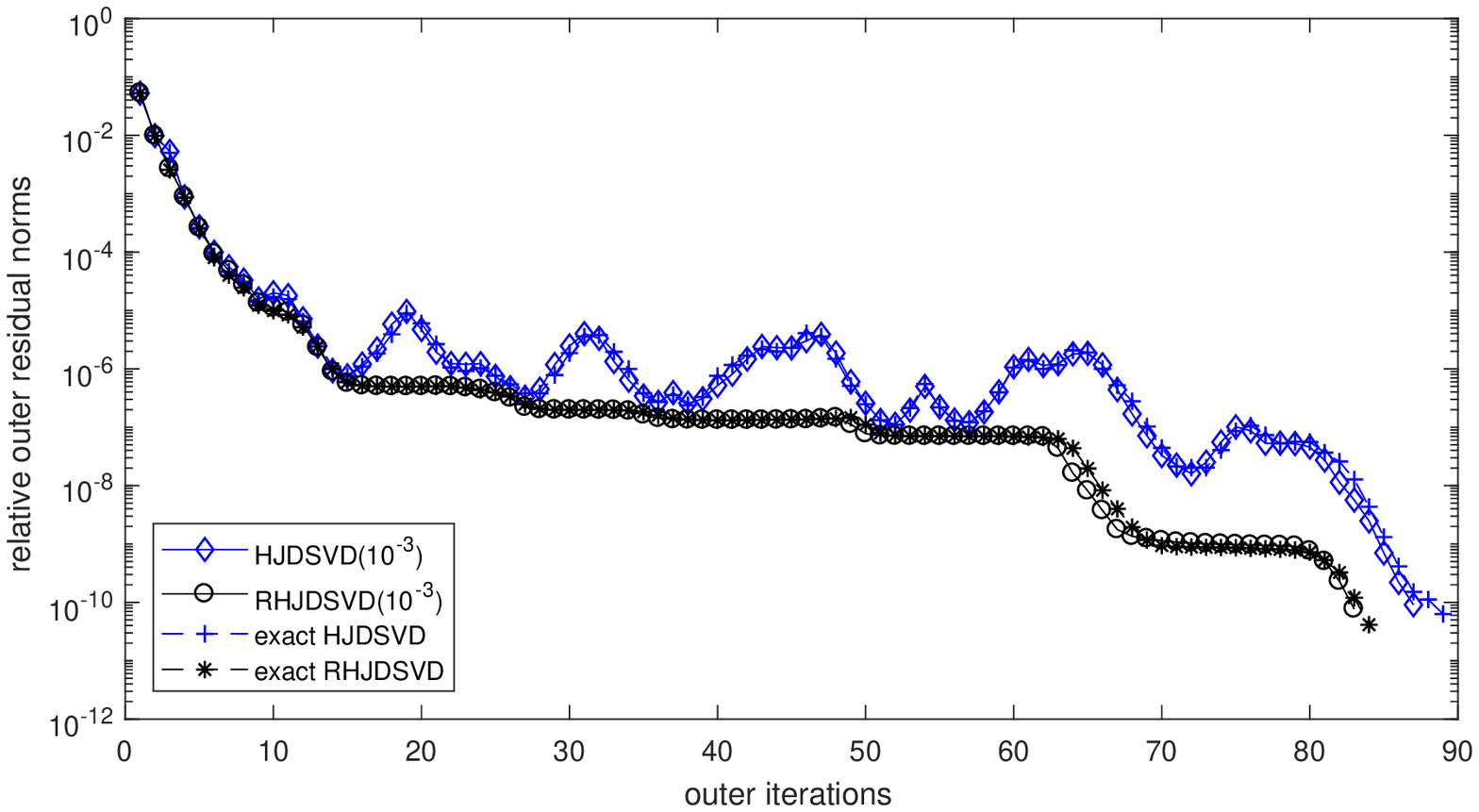}}
\ \ \ \
\subfloat[]{\label{fig2b}\includegraphics[width=0.48\textwidth]{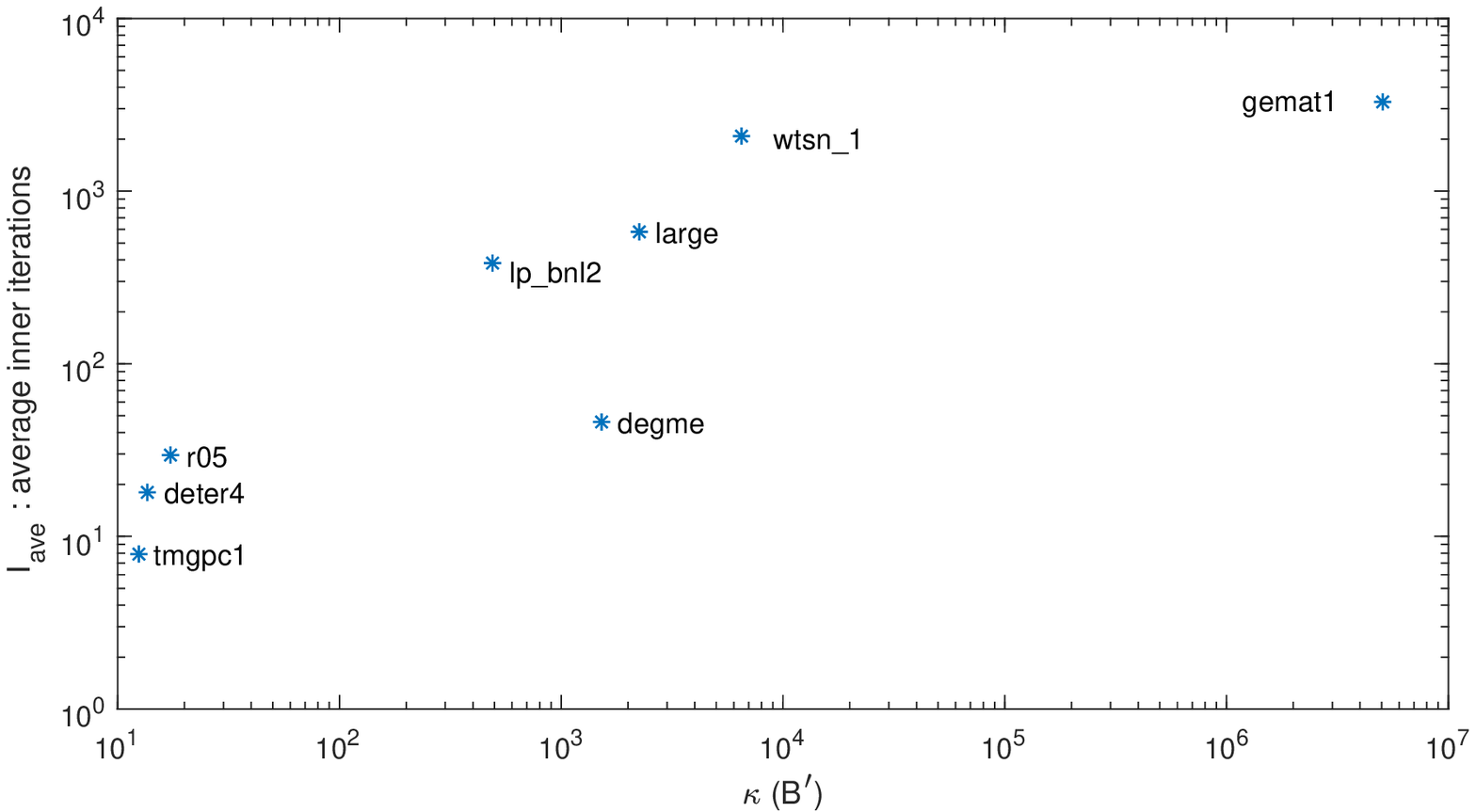}}
\caption{(a) lp\_bnl2 with $\tau=8.16$;
   (b)$I_{ave}$ of HJDSVD($10^{-3}$) versus $\kappa(B^{\prime})$ in Table \ref{table0}.}
\label{fig2}
\end{figure}
\begin{table}[tbhp]
{\footnotesize
\caption{lp\_bnl2 with $\tau=8.16$.}\label{table2}
\begin{center}
\begin{tabular}{|c|c|c|c|c|c|c|c|} \hline
\multicolumn{4}{|c|}{ Accuracy: $\widetilde\varepsilon=10^{-3}$} &\multicolumn{4}{|c|} { Accuracy: Iter. Exact}    \\ \hline
Algorithm      &$I_{out}$     &$I_{in}$ &$T_{cpu}$    &Algorithm     &$I_{out}$  &$I_{in}$ &$T_{cpu}$      \\ \hline
\mbox{HJDSVD}      &87         &30210       &5.11       &\mbox{HJDSVD}      &89         &122495     &21.7       \\
\mbox{RHJDSVD}     &83         &22620       &4.16        &\mbox{RHJDSVD}     &84        &113862      &20.3      \\ \hline
\multicolumn{4}{|c|}{ Accuracy: $\widetilde\varepsilon=10^{-4}$} &\multicolumn{4}{|c|}{ Accuracy: LU factorization}\\ \hline
Algorithm      &$I_{out}$     &$I_{in}$ &$T_{cpu}$    &Algorithm     &$I_{out}$  &$I_{in}$  &$T_{cpu}$     \\ \hline
\mbox{HJDSVD}      &87         &30044       &5.41       &\mbox{HJDSVD}      &89         &$-$         &1.70       \\
\mbox{RHJDSVD}     &83         &23209       &4.24       &\mbox{RHJDSVD}     &84         &$-$         &1.53        \\ \hline
\end{tabular}
\end{center}
}
\end{table}

The results are displayed in Table \ref{table2} and Figure \ref{fig2} (a).
We observe from Figure \ref{fig2} (a) that both
non-restarted and restarted HJDSVD($10^{-3}$) and RHJDSVD($10^{-3}$)
behave almost the same as the corresponding exact HJDSVD and RHJDSVD,
respectively, and RHJDSVD converges more smoothly and gives
more accurate approximate singular triplet than HJDSVD.
We also see from Table \ref{table2} that for
$\widetilde\varepsilon=10^{-3}$ and $10^{-4}$, the inexact JDSVD
algorithms use almost the same outer iterations to
converge as their exact counterparts do.
As far as the total inner iterations and overall efficiency are concerned,
we can see from Table~\ref{table2} that all our inexact JDSVD algorithms
reduce more than $75\%$ of total inner iterations and CPU
time than their iterative exact counterparts.
Apparently, a smaller $\widetilde \varepsilon$ is unnecessary.
A final note is that RHJDSVD performs better than HJDSVD in terms of
both outer and inner iterations.

\emph{We now test the other matrices. For
the three middle-scale matrices $A_1=$ r05 with $\tau_1=4.75$,
$A_2=$ large  with $\tau_2=9.85$, and $A_3=$ gemat with $\tau_3=14.4$,
the desired singular values $\sigma^*\approx 3.43$, $\sigma^*\approx 8.06$
and $\sigma^*\approx 14.38$ are all interior ones of the
test matrices
and are highly clustered with some other singular values;
for the three large-scale matrices
$A_4=$ tomographic1 ('tmgpc1') with $\tau_4=8$,
$A_5=$ watson\_1 ('wtsn\_1') with $\tau_5=14$,
and $A_6=$ degme with $\tau_6=5.56$,
the desired $\sigma^*\approx 6.98$ is the largest one of $A_4$,
$\sigma^*\approx 14.004$ is an interior one of $A_5$,
and $\sigma^*\approx 4.13$ is the smallest one of $A_6$.}

\begin{table}[tbhp]
{\footnotesize
\caption{Results on the other test matrices.}\label{table3}
\begin{center} \resizebox{\textwidth}{23mm}{
\begin{tabular}{|c|c|ccc|ccc|ccc|} \hline
\multirow{2}{*}{Matrix}&\multirow{2}{*}{Algorithm}
&\multicolumn{3}{c|}{$\widetilde\varepsilon=10^{-3}$}
&\multicolumn{3}{c|}{$\widetilde\varepsilon=10^{-4}$}
&\multicolumn{3}{c|}{Iter. Exact}\\ \cline{3-11}
& &$I_{out}$ &$I_{in}$ &$T_{cpu}$ &$I_{out}$ &$I_{in}$ &$T_{cpu}$ &$I_{out}$ &$I_{in}$ &$T_{cpu}$ \\ \hline
\multirow{2}{*}{r05 }
&HJDSVD  &89 &2626 &1.56   &89 &2716 &1.61     &89 &14733 &7.19  \\ %\cline{2-9}
&RHJDSVD &82 &2100 &1.24   &76 &2010 &1.17     &80 &12981 &6.05  \\ \hline
\multirow{2}{*}{large }
&HJDSVD  &147 &85865 &25.5  &145 &88449 &25.5   &154 &492011 &1.40e+2  \\ %\cline{2-9}
&RHJDSVD &151 &75227 &21.7  &146 &75820 &21.9   &151 &468380 &1.36e+2  \\ \hline
\multirow{2}{*}{gemat1}
&HJDSVD  &14 &47801 &20.2   &13 &50305 &17.6    &13 &89327 &35.9  \\ %\cline{2-9}
&RHJDSVD &14 &48205 &19.7   &14 &54638 &22.1    &13 &86521 &35.8  \\ \hline
\multirow{2}{*}{tmgpc1}
&HJDSVD  &22 &176 &1.70 &22 &185 &1.63  &22 &1085  &7.55 \\ %\cline{2-9}
&RHJDSVD &22 &174 &1.92 &22 &183 &1.69  &22 &1082  &7.25  \\ \hline
\multirow{2}{*}{wstn\_1}
&HJDSVD  &13 &27181 &8.65e+2   &13 &30033  &9.68e+2   &13 &44306 &1.94e+3  \\ %\cline{2-9}
&RHJDSVD &13 &26989 &8.73e+2   &13 &30017  &9.65e+2    &13 &44462 &1.94e+3  \\ \hline
\multirow{2}{*}{degme}
&HJDSVD  &11 &509 &37.3 &10 &568    &41.3   &10 &2395 &1.67e+2  \\ %\cline{2-9}
&RHJDSVD &11 &510 &36.2 &10 &564    &40.6   &10 &2387 &2.23e+2  \\ \hline
\end{tabular}}
\end{center}
}
\end{table}

Since the matrices $A_i,\ i=4,5,6,$  are very large in our computer,
it is unaffordable to implement the LU factorization of $\begin{bmatrix}
\begin{smallmatrix}-\tau I_M&A_i\\A_i^T &-\tau I_N \end{smallmatrix}\end{bmatrix}$.
So we only use MINRES to solve the correction
equations involved in the exact JDSVD algorithms.
For $A_i, \ i=1,2,3$, we find that the outer iterations used
by the exact JDSVD algorithms,
where the correction equations are solved by the LU factorization of
$\begin{bmatrix}\begin{smallmatrix}-\tau I_M&A_i\\A_i^T &-\tau I_N
\end{smallmatrix}\end{bmatrix}$,
are exactly the same as those by the iterative exact JDSVD algorithms, as they
should be. Table~\ref{table3} lists the details.

For these six matrices, we have observed very similar phenomena to the previous
examples, so we make comments on them together. We find that all the
inexact JDSVD type algorithms behave almost the same as their iterative exact
counterparts and use almost the same or very comparable outer iterations as
the latter ones do. It is seen from Table~\ref{table3} that
when computing the desired interior singular triplets of $A_3$ and $A_5$,
our inexact JDSVD algorithms reduce more than $32\%$ of
total inner iterations and $38\%$ of CPU time,
compared with their iterative exact counterparts.
The reductions of total inner iterations are up to $76\%$
for $A_6$, and $81\%$ for $A_1$, $A_2$ and $A_4$,
and the corresponding reductions of total CPU time are up to
$95\%$ and $73\%$, respectively. These are
substantial savings, compared with their iterative exact counterparts.

Clearly, a fairly small $\widetilde\varepsilon\in[10^{-4},10^{-3}]$ is
enough for the
inexact JDSVD algorithms to mimic their exact counterparts and reduce
the computational costs substantially, and a smaller
$\widetilde\varepsilon$ is unnecessary.

Summarizing the previous experiments, we conclude that HJDSVD and RHJDSVD
are suitable for computing both an interior and an extreme
singular triplet.

In the following we compute the five singular values of the matrices
in Table \ref{table0} nearest to the given target $\tau$'s
given in the previous experiments and the corresponding left and right
singular vectors.

Table~\ref{table7} gives the results on the eight test matrices.
Figure \ref{fig3} depicts the convergence curves of the inexact and exact
JDSVDs of all cycles for computing the five singular triplets of deter4
with $\tau=7$, lp\_bnl2 with $\tau=8.16$,
r05 with $\tau=4.75$ and large with $\tau=9.85$.
Notice that each algorithm computes the desired singular triplets successively
and it computes the next one after the previous one has converged.
As a result, its convergence curve has five stages and contains five convergence
points (valleys), and each stage computes one singular triplet.

We see that the restarted HJDSVD($10^{-3}$)
and RHJDSVD($10^{-3}$) behave quite like their exact counterparts.
For all the test matrices, we see from Table \ref{table7} that
our inexact JDSVD type algorithms use very comparable outer iterations
to their iterative exact counterparts. Furthermore,
for the matrices gemat1 and wstn\_1,
our inexact JDSVD algorithms reduce more than $34\%$ and $22\%$
of total inner iterations, respectively,
compared with their exact counterparts;
The reductions of total inner reductions
for deter4, lp\_bnl2, r05, large, tmgpc1 and degme are more than $75\%$,
substantial savings, compared with the exact JDSVD algorithms.
Clearly, our inexact JDSVD algorithms with deflation can mimic their exact
counterparts well with a fairly small $\widetilde \varepsilon\in[10^{-4},10^{-3}]$,
and meanwhile reduce the computational cost substantially.

In addition, from Figure \ref{fig3} we have observed faster and smoother
convergence of the RHJDSVD algorithms than the HJDSVD algorithms when computing each of
the desired singular triplets. We can also
see from Table \ref{table7} that the total outer iterations
used by RHJDSVD($\widetilde\varepsilon$) and
HJDSVD($\widetilde\varepsilon$) are very comparable, but for most of
the matrices, i.e., deter4, lp\_bnl2, r05, large, tmgpc1
and degme, RHJDSVD($\widetilde\varepsilon$) uses fewer total inner iterations
than HJDSVD($\widetilde\varepsilon$), so does the CPU time. Therefore,
As a whole, we conclude that, when computing more than one singular triplet,
RHJDSVD is generally more robust and efficient than HJDSVD.

\begin{figure}[tbhp]
\centering
\subfloat[deter4 with $\tau=7$.]
{\label{fig3a}\includegraphics[width=0.48\textwidth]{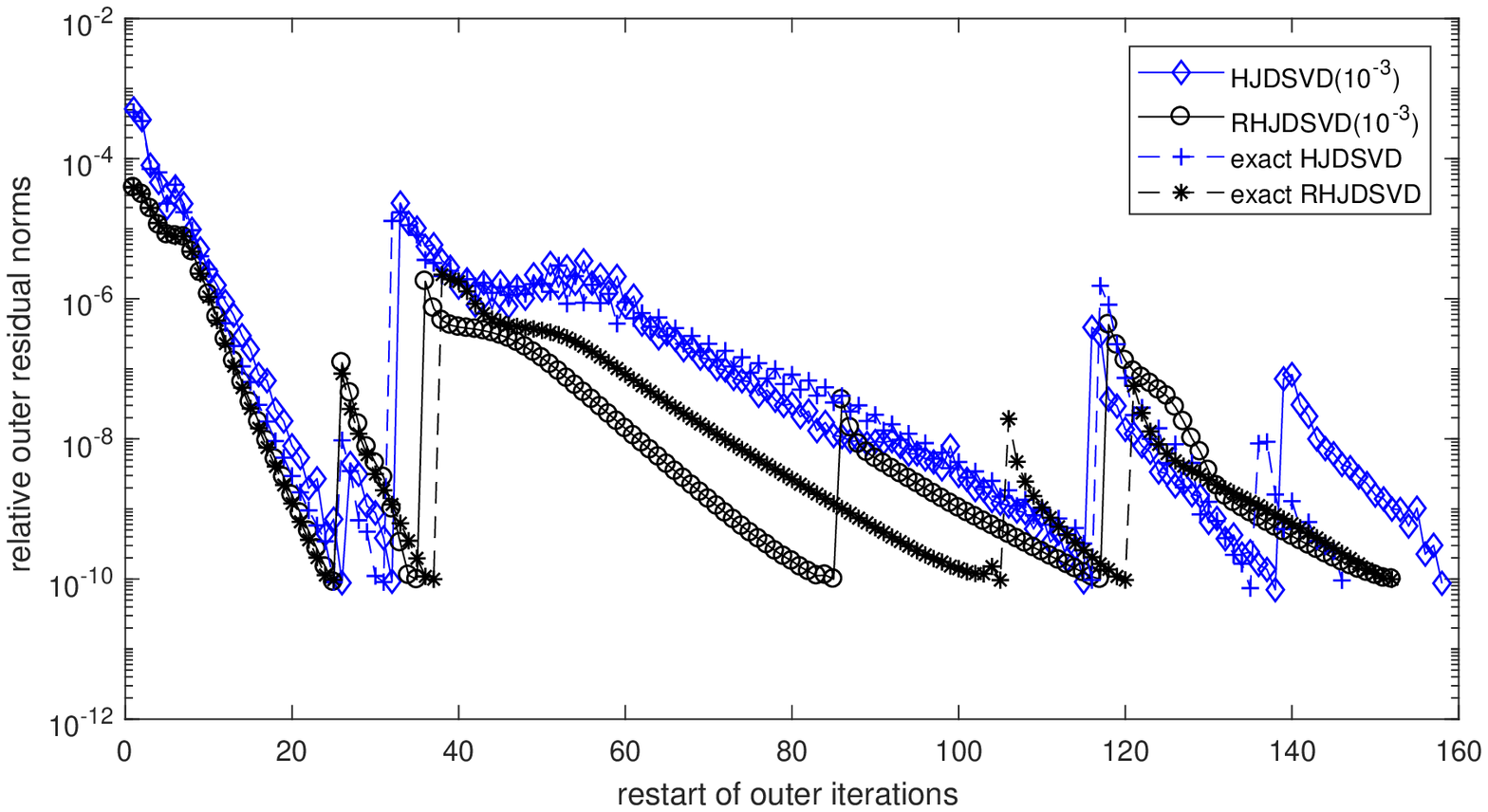}}
\ \ \ \
\subfloat[lp\_bnl2 with $\tau=8.16$.]
{\label{fig3b}\includegraphics[width=0.48\textwidth]{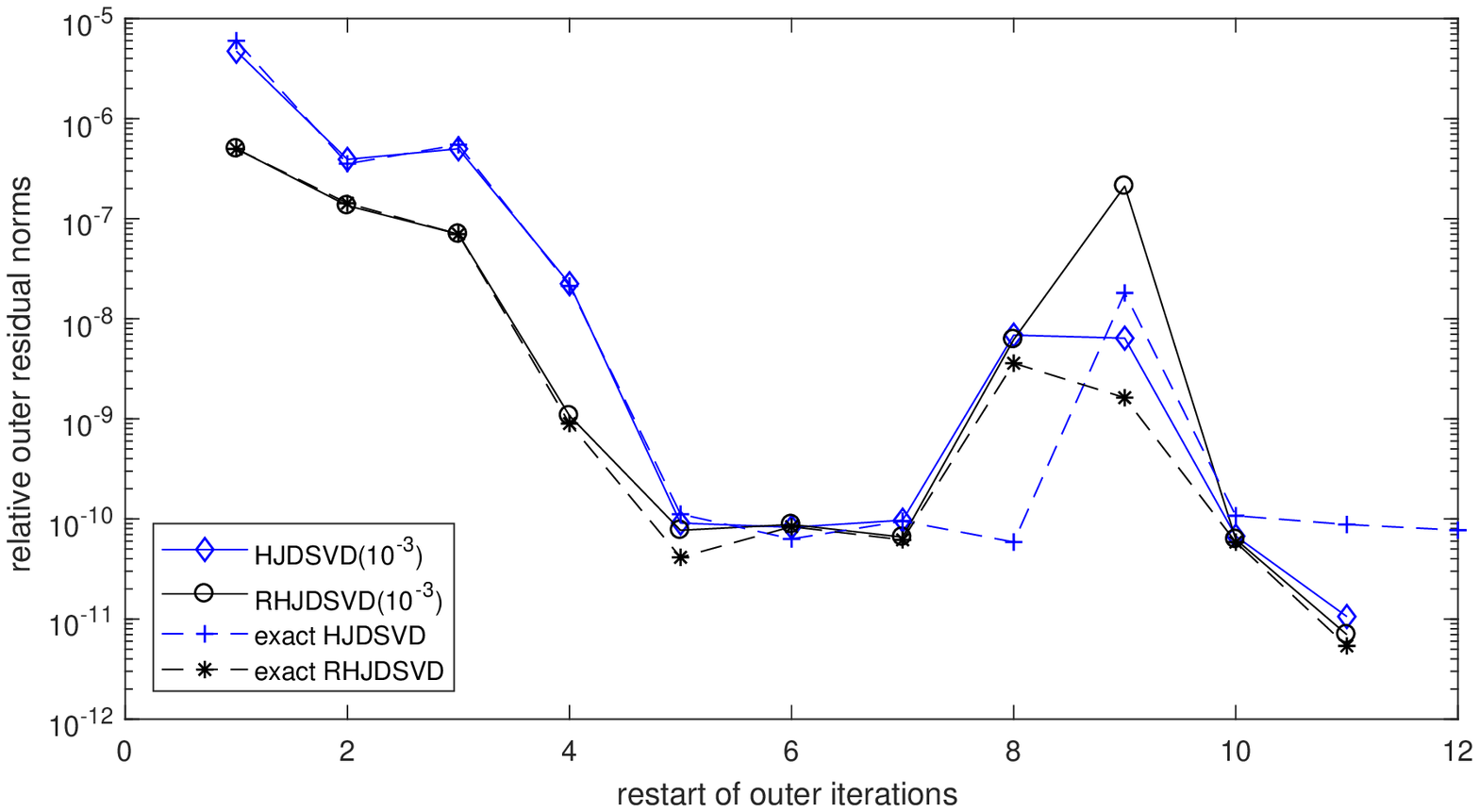}}\\
\subfloat[r05 with $\tau=4.75$. ]
{\label{fig3c}\includegraphics[width=0.48\textwidth]{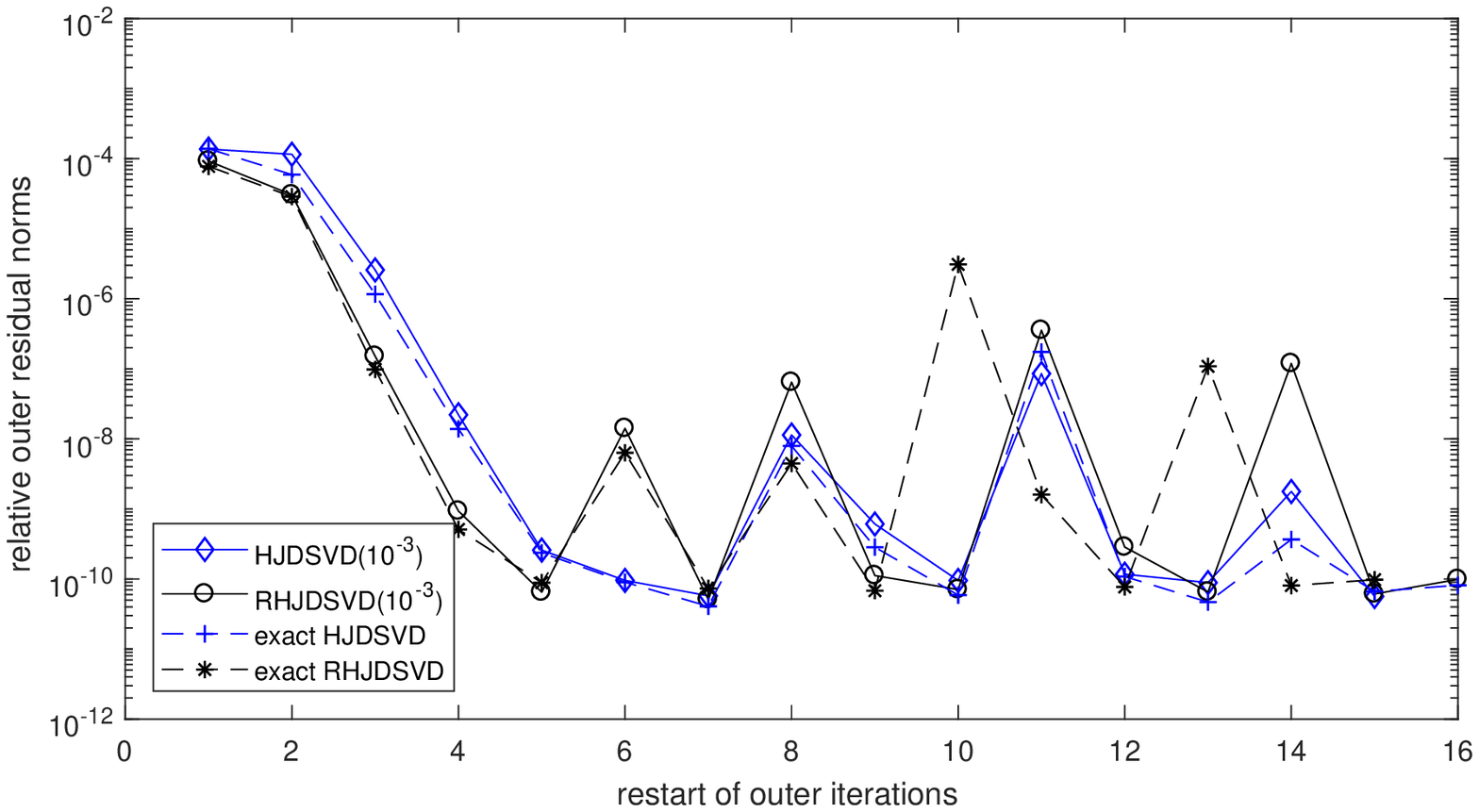}}
\ \ \ \
\subfloat[large with $\tau=9.85$.]
{\label{fig3d}\includegraphics[width=0.48\textwidth]{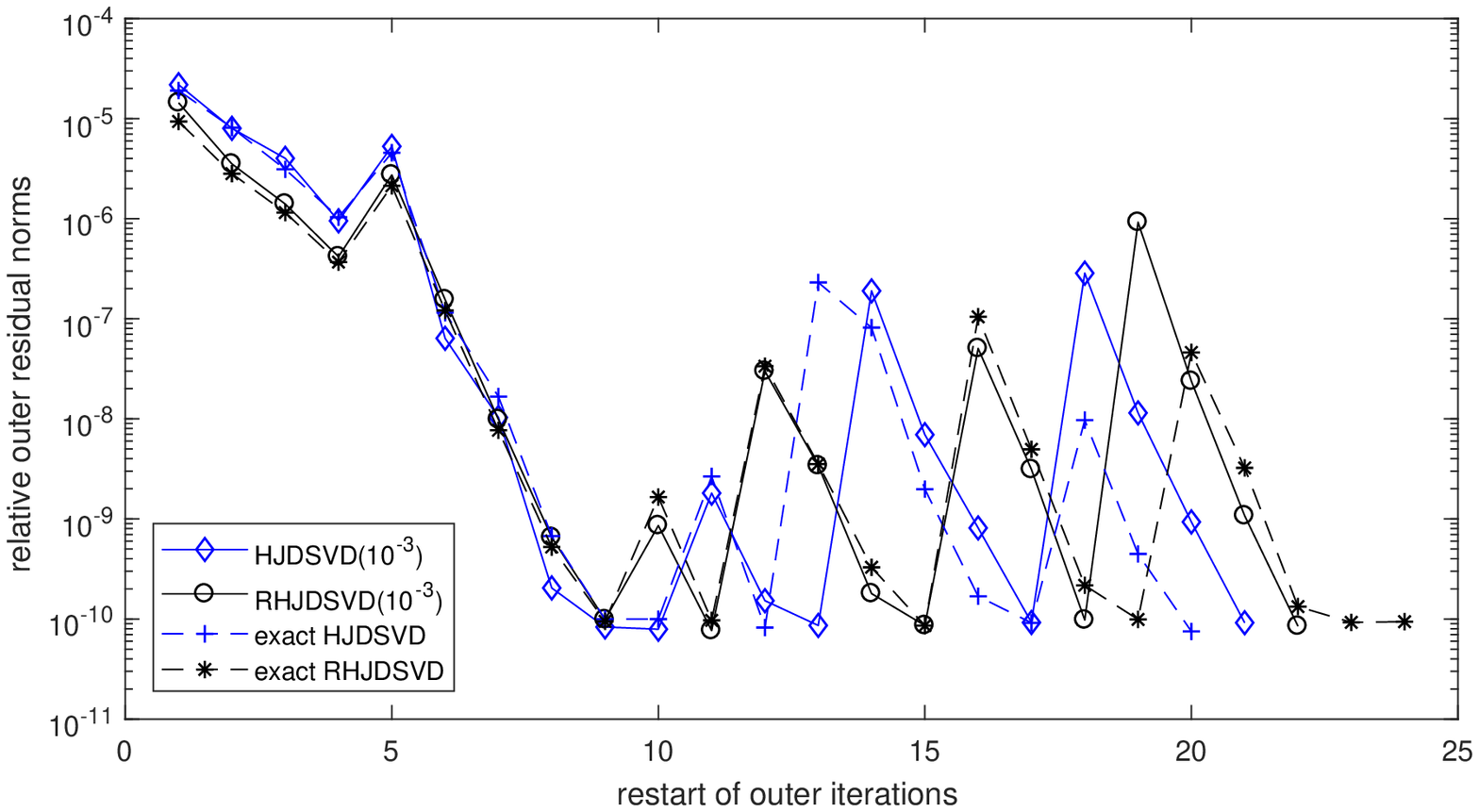}}
\caption{Computing the five singular triplets of the matrices with a given target $\tau$.}
\label{fig3}
\end{figure}

\begin{table}[tbhp]
{\footnotesize
\caption{Results on the computation of the five singular triplets of the
test matrices in Table \ref{table0}.}\label{table7}
\begin{center} \resizebox{\textwidth}{28mm}{
\begin{tabular}{|c|c|ccc|ccc|ccc|} \hline
\multirow{2}{*}{Matrix}&\multirow{2}{*}{Algorithm}
&\multicolumn{3}{c|}{$\widetilde\varepsilon=10^{-3}$}
&\multicolumn{3}{c|}{$\widetilde\varepsilon=10^{-4}$}
&\multicolumn{3}{c|}{Iter. Exact}\\ \cline{3-11}
& &$I_{out}$ &$I_{in}$ &$T_{cpu}$ &$I_{out}$ &$I_{in}$ &$T_{cpu}$ &$I_{out}$ &$I_{in}$ &$T_{cpu}$ \\ \hline
\multirow{2}{*}{deter4 }
&HJDSVD  &2610  &47045  &20.1    &1897  &34374  &14.3      &2324  &190378  &63.2   \\ %\cline{2-9}
&RHJDSVD &2508  &29669  &14.4    &2377  &28289  &13.7      &2380  &192268  &64.8   \\ \hline
\multirow{2}{*}{lp\_bnl2}
&HJDSVD  &114  &39063  &7.02    &122  &41427  &7.75      &134  &180803  &33.8   \\ %\cline{2-9}
&RHJDSVD &122  &34391  &6.64    &114  &32221  &6.09      &114  &152770  &28.4   \\ \hline
\multirow{2}{*}{r05}
&HJDSVD  &186  &5515   &3.36    &192  &5724   &3.41      &184  &30264   &15.2   \\ %\cline{2-9}
&RHJDSVD &192  &4969   &3.28    &176  &4629   &2.92      &177  &28521   &14.0   \\ \hline
\multirow{2}{*}{large}
&HJDSVD  &289  &170847  &54.5    &285  &172054  &53.6      &273  &874840  &2.65e+2   \\ %\cline{2-9}
&RHJDSVD &311  &160411  &50.2    &300  &156601  &48.7      &316  &988726  &3.05e+2   \\ \hline
\multirow{2}{*}{gemat1}
&HJDSVD  &27  &82899  &37.0    &25  &78417  &33.7      &23  &149002  &63.8   \\ %\cline{2-9}
&RHJDSVD &31  &94866  &41.5    &27  &88256  &37.6      &23  &145875  &64.3   \\ \hline
\multirow{2}{*}{tmgpc1}
&HJDSVD  &80  &589  &6.53    &78  &582  &6.58      &75  &3525  &28.1   \\ %\cline{2-9}
&RHJDSVD &77  &566  &6.81    &75  &561  &6.16      &75  &3531  &27.5   \\ \hline
\multirow{2}{*}{wstn\_1}
&HJDSVD  &25  &42834  &1.47e+3    &22  &41995  &1.44e+3      &18  &59750  &2.76e+3   \\ %\cline{2-9}
&RHJDSVD &25  &44491  &1.54e+3    &22  &41933  &1.58e+3      &17  &56947  &2.50e+3   \\ \hline
\multirow{2}{*}{degme}
&HJDSVD  &57  &2445  &2.03e+2    &58  &2749  &2.28e+2      &55  &13323  &1.12e+3   \\ %\cline{2-9}
&RHJDSVD &56  &2173  &1.80e+2    &56  &2425  &2.01e+2      &52  &12377  &1.28e+3   \\ \hline
\end{tabular}}
\end{center}
}
\end{table}

We now get insight into the role that $\kappa(B^{\prime})$ plays
in the inexact JDSVD algorithms. We take HJDSVD($10^{-3}$)
for computing one singular triplet as an example.
Denote by $I_{ave}=I_{in}/I_{out}$ the average inner iterations
per outer iteration. Note that a smaller $I_{ave}$ indicates a faster,
on average, convergence of MINRES. To make it clearer so as to see how
$\kappa(B^{\prime})$ influences the convergence speed of inner iterations,
we mark the matrix names in the plot of $\kappa(B^{\prime})$ versus $I_{ave}$
in Figure~\ref{fig2} (b). We observe a trend from the figure and Table~\ref{table0}
that the larger $\kappa(B^{\prime})$ is, the more
inner iterations per outer iteration are needed to achieve
the convergence, i.e., the more slowly MINRES converges.
We have observed similar phenomena for the exact and inexact HJDSVD
and RHJDSVD, which confirms our analysis at the end of Section \ref{section:4}.

Finally, let us make further comments on the correction
equations \eqref{equation:1} and \eqref{correction}.
This system is typically symmetric indefinite,
and it may be ill conditioned when
the desired singular value $\sigma$ is an interior one,
so that MINRES or other Krylov
iterative solvers may converge slowly.
Therefore, preconditioning is naturally appealing.
However, it is hard to obtain an effective preconditioner for MINRES when
(\ref{equation:1}) or \eqref{correction} is highly indefinite and ill conditioned.
Using the MATLAB function {\sf ilu.m}, we have tried the sparse incomplete LU factorizations
of $\begin{bmatrix}\begin{smallmatrix}-\tau I_{M} &A\\A^T &-\tau I_{N} \end{smallmatrix}
\end{bmatrix}$ with $setup.droptol=0.1$ and $0.01$.
With such preconditioners, the preconditioned correction equations are nonsymmetric,
and we use the Krylov solver BiCGStab algorithm \cite{saad03} to solve them.
Unfortunately, we have found that such preconditioners does not work effectively.
For most of the test problems, the preconditioned BiCGStab is even not competitive
with the unpreconditioned MINRES and uses more inner iterations.
So we do not report the results on the preconditioned BiCGStab.

\section{Conclusions}\label{section:6}

We have proposed harmonic and refined harmonic JDSVD methods
for computing several singular triplets of a large matrix $A$.
By a rigorous one-step analysis, we have proved for the first time that, provided
the correction equations \eqref{equation:1} and \eqref{correction}
involved in the JDSVD methods are only solved
with low or modest accuracy, the inexact JDSVD methods mimic their inexact
counterparts well.
Based on the theory, we have proposed general-purpose practical
stopping criteria for inner iterations involved in the two inexact JDSVD type
methods. We should point out that the theory and criteria also work for
the standard and refined inexact JDSVD methods, but
these methods are inferior to HJDSVD and RHJDSVD for computing interior
singular triplets. In the meantime, our results apply to the methods
in \cite{hochstenbach2001jacobi,hochstenbach2004harmonic}.

Numerical experiments have confirmed our theory. We have tested a number of problems
and compared the non-restarted and restarted
inexact JDSVD algorithms with their exact counterparts.
We have found that the inexact JDSVD algorithms indeed mimic
the exact JDSVD algorithms
very well when the correction equations are solved with low or modest accuracy
$10^{-4}\sim 10^{-3}$.
A great advantage of the inexact JDSVD type algorithms is that they
reduce the computational cost very substantially, compared with their
iterative exact versions. Furthermore, the experiments have illustrated that
RHJDSVD generally outperforms HJDSVD.

\section*{Acknowledgements} We thank the referee for his/her careful reading of
the paper and valuable comments that made us
improve the presentation of the paper substantially.

%\bibliographystyle{siamplain}
%\bibliography{references}

\end{document}